\documentclass[10pt]{amsart}

\numberwithin{equation}{section}
\newcounter{mnote}
\let\oldmarginpar\marginpar
\renewcommand\marginpar[1]{\-\oldmarginpar[\raggedleft\footnotesize #1]%
	{\raggedright\footnotesize #1}}

\usepackage[english]{babel}
\usepackage{amssymb}
\usepackage{mathrsfs}
\usepackage{stmaryrd}
\usepackage{chemarrow}
\usepackage[norelsize]{algorithm2e}
\usepackage{enumerate}
\usepackage{graphicx}
\usepackage{subcaption}
\usepackage[all]{xy}
\usepackage{tikz}
\usetikzlibrary{arrows}
\usepackage{multirow}
\usepackage[colorinlistoftodos]{todonotes}
\usepackage[colorlinks=true, allcolors=blue]{hyperref}
\usepackage[hyperpageref]{backref} % 参考文献中backref
\usepackage{hyperref}
\hypersetup{
	colorlinks=true,
	linkcolor=blue,
	anchorcolor=blue,
	citecolor=red}
\usepackage{verbatim}
\usepackage{booktabs}
\allowdisplaybreaks

\newtheorem{theorem}{Theorem}[section]
\newtheorem{lemma}[theorem]{Lemma}

\newtheorem{example}[theorem]{Example}

\newtheorem{remark}[theorem]{Remark}
\newtheorem{assumption}[theorem]{Assumption}

\newcommand{\normmm}[1]{{\left\vert\kern-0.25ex\left\vert\kern-0.25ex\left\vert #1
		\right\vert\kern-0.25ex\right\vert\kern-0.25ex\right\vert}}

\renewcommand{\div}{\operatorname{div}}
\newcommand{\grad}{\operatorname{grad}}
\newcommand{\tr}{\operatorname{tr}}

\newcommand{\dev}{\operatorname{dev}}

\begin{document}

% \title[]{FEM FOR SGE MODEL}

% \author{XUEHAI HUANG \quad ZHEQIAN TANG}
%\address{}

\title[Lowest-Order Robust Mixed FEM for SGE]{A Lowest-Order Robust Mixed Finite Element Method with a Third-Order Tensor Variable for Strain Gradient Elasticity}
\author{Xuehai Huang}%
\address{School of Mathematics, Shanghai University of Finance and Economics, Shanghai 200433, China}%
\email{huang.xuehai@sufe.edu.cn}%
\author{Zheqian Tang}%
\address{School of Mathematics, Shanghai University of Finance and Economics, Shanghai 200433, China}%
\email{tangzq0329@163.com}%

\thanks{%The second author is the corresponding author. 
%The first author was partially supported by NSFC (Grant No.\ 12071289) and the National Key Research and Development Project (2020YFA0709800).
%The second author was supported by the National Natural Science Foundation of China (Grant No.\ 12171300). %, and the Natural Science Foundation of Shanghai  (Grant No.\ 21ZR1480500).
}
%%%%% Keywords %%%%%%%%%%%
%\keywords{ }
\makeatletter
\@namedef{subjclassname@2020}{\textup{2020} Mathematics Subject Classification}
\makeatother

\begin{abstract}
%An optimal and robust lowest-order mixed finite element method is developed for the strain gradient elasticity (SGE) model in arbitrary dimensions, achieving optimal convergence without the need for Nitsche’s technique. The method employs lowest-order $H(\div)$-conforming elements for third-order tensors to discretize the stress gradient $\grad\boldsymbol{\sigma}(\boldsymbol{u})$ and Crouzeix--Raviart elements for the displacement $\boldsymbol{u}$. Optimal error estimates, uniform with respect to both the size parameter $\iota$ and the Lam\'{e} coefficient $\lambda$, are also established  based on an $L^2$-bounded interpolation operator for third-order tensors. Finally, numerical experiments in both two and three dimensions are presented to validate the theoretical results.
A lowest-order mixed finite element method is developed for the strain gradient elasticity (SGE) model in arbitrary dimensions. We take the physically meaningful third-order double stress tensor $\boldsymbol{\Phi}:=\iota^2\grad\boldsymbol{\sigma}(\boldsymbol{u})\in\mathbb{S}\otimes\mathbb{R}^d$ as a primary variable and derive a distributional mixed formulation. The double stress is approximated by an $\mathbb{S}\otimes\mathbb{R}^d$-valued extension of the lowest-order Raviart--Thomas element, while the displacement is approximated by the vector-valued linear Crouzeix--Raviart element. Thus, the method avoids both high-degree bubble enrichment and a Nitsche-type treatment of the higher-order boundary condition. We establish parameter-robust discrete stability, an optimal first-order error estimate for fixed parameters, and a complementary parameter-uniform $\mathcal{O}(\iota^{1/2}+h)$ error estimate with constants independent of both the size parameter $\iota$ and the Lam\'e coefficient $\lambda$. In the boundary-layer regime $\iota^{1/2}\lesssim h$, the latter retains a first-order convergence rate in $h$. We also develop a local quadratic post-processing and a hybridized formulation. Numerical experiments in two and three dimensions support the theoretical results.
\end{abstract}
\maketitle

%%%%%%%%%%%%%%%%%%%%%%%%%%%%%%%%%%%%%%%%
%% sect-introduction
%%%%%%%%%%%%%%%%%%%%%%%%%%%%%%%%%%%%%%%%
\section{Introduction}\label{sec1}
Classical elasticity theory may fail to describe the mechanical behavior of materials at small scales, where microstructural interactions and size effects become significant
\cite{toupin1962elastic,mindlin1964micro,aifantis1984microstructural}.
Strain gradient elasticity (SGE) theories account for these effects by incorporating higher-order deformation gradients and additional material parameters. The Toupin--Mindlin theory
\cite{toupin1962elastic,mindlin1964micro,mindlin1968first}
provides a general framework involving multiple intrinsic length scales. A simplified SGE model with only one additional size parameter was subsequently proposed in
\cite{altan1992structure,ru1993simple}
and has been widely used in engineering applications
\cite{askes2011gradient,polizzotto2015unifying}.

In this paper, we consider the SGE model given in \cite{altan1992structure,ru1993simple}:
\begin{equation}\label{SGE0}
\begin{cases}
-\div ((\boldsymbol{I}-\iota^{2}\Delta)\boldsymbol{\sigma}(\boldsymbol{u}))=\boldsymbol{f} &\mbox{in} \ \Omega,\\
\boldsymbol{u}=\partial_{n}\boldsymbol{u}=0 &\mbox{on} \ \partial\Omega,
\end{cases}
\end{equation}
where $\Omega\subset\mathbb{R}^d (d\geq 2)$ is a bounded polytopal domain, $\boldsymbol{f}\in L^2(\Omega;\mathbb R^d)$ is the applied force, $\boldsymbol{u}=(u_1, \ldots, u_d)^\intercal$ is the displacement field, $\partial_{n}\boldsymbol{u}$ is the normal derivative of $\boldsymbol{u}$, the stress $\boldsymbol{\sigma}(\boldsymbol{u})=2\mu\boldsymbol{\varepsilon}(\boldsymbol{u})+\lambda(\div\boldsymbol{u})\boldsymbol{I}$, and $\boldsymbol{\varepsilon}(\boldsymbol{u})=(\varepsilon_{ij}(\boldsymbol{u}))_{d\times d}$ is the strain tensor field with $\varepsilon_{ij}(\boldsymbol{u})=(\partial_j u_i+\partial_i u_j)/2$. Here $\mu$ and $\lambda$ are the Lam\'{e} constants, $\boldsymbol{I}$ is the identity tensor field, and $\iota\in(0,1]$ denotes the size parameter of the material under discussion.

The SGE model \eqref{SGE0} is a fourth-order strain gradient perturbation of linear elasticity. As $\iota$ approaches zero, it reduces to the second-order linear elasticity problem, whose solution generally fails to satisfy the additional boundary condition $\partial_n\boldsymbol{u}=0$; boundary layers consequently arise. Furthermore, as $\lambda$ tends to infinity, the material becomes nearly incompressible, which may cause volumetric locking for low-order conforming finite element methods. The fourth-order structure and clamped boundary conditions, together with these two parameter limits, make the construction of genuinely low-order methods that are stable and parameter robust particularly challenging.

Several conforming, nonconforming, and mixed finite element methods have been developed for strain gradient elasticity; see, for instance,
\cite{zervos2001modelling,torabi2018c1,MR3712289,MR4296093,MR4021023,chirkov2024mixed-1,chirkov2025mixed}.
However, many existing methods either do not provide a complete error analysis or do not yield estimates uniform with respect to $\lambda$.
Robust finite element analysis with respect to both $\iota$ and $\lambda$ was initiated in
\cite{ChenHuangHuang2023,MR4549866},
where sharp but suboptimal convergence rates of
$\mathcal{O}(h^{1/2})$ were obtained.
In the boundary-layer regime $\iota^{1/2}\lesssim h$, the parameter-uniform estimate proved below retains a first-order convergence rate in $h$, thereby overcoming the loss of convergence order observed in previous robust analyses.

To achieve optimal convergence in the presence of boundary layers, the recently developed optimal-order methods in \cite{ChenHuangHuang2025,HuangHuangTang2025} employ Nitsche's technique \cite{MR0341903,MR2917211} to handle the higher-order boundary condition.
More precisely, a quartic nonconforming finite element method in two dimensions was developed in \cite{ChenHuangHuang2025}, while a quadratic nonconforming finite element method in arbitrary dimensions was proposed in \cite{HuangHuangTang2025}. Both methods are based on nonconforming finite element discretizations of the short complex
\begin{equation}\label{complexdivpartH2}
	H_0^2(\Omega;\mathbb{R}^d)
	\xrightarrow{\div}
	H_0^1(\Omega)\cap L_0^2(\Omega)
	\xrightarrow{}0.
\end{equation}
However, constructing a nonconforming finite element discretization of \eqref{complexdivpartH2} is technically demanding. The local spaces in these methods require enrichment by high-degree bubble functions, namely, polynomials of degree $9$ in \cite{ChenHuangHuang2025} and of degree $3d+1$ in \cite{HuangHuangTang2025}. These constructions underline the difficulty of obtaining genuinely low-order SGE methods with both optimal convergence and parameter robustness.

The purpose of this paper is to develop a genuinely lowest-order mixed finite element method for problem~\eqref{SGE0} in arbitrary dimensions. A central feature of our approach is to take the physically meaningful third-order double stress tensor
\[
\boldsymbol{\Phi}
:=\iota^2\grad\boldsymbol{\sigma}(\boldsymbol{u})
\]
as a primary mixed variable. The double stress is associated with the strain-gradient effects in the model; it is an $\mathbb{S}\otimes\mathbb{R}^d$-valued third-order tensor that is symmetric in its first two indices. Thus, the novelty here is not the physical notion of double stress itself, but its use as a primary variable in a distributional mixed formulation. This gives a route fundamentally different from the Nitsche-based approaches above: rather than imposing the higher-order boundary condition through Nitsche boundary terms, we incorporate it into the mixed variational structure. Specifically, we seek $(\boldsymbol{\Phi},\boldsymbol{u})
\in
H^{-1}(\div\div,\Omega;
\mathbb{S}\otimes\mathbb{R}^d)
\times H_0^1(\Omega;\mathbb{R}^d)$
such that
\begin{subequations}\label{intro-mix}
	\begin{align}
			\label{intro-mix1}
			\iota^{-2}(\widetilde{\mathcal{A}}\boldsymbol{\Phi},\boldsymbol{\Psi})-\langle\div\div\boldsymbol{\Psi}, \boldsymbol{u}\rangle &= 0 \quad\quad\quad\;\; \forall \ \boldsymbol{\Psi}\in  H^{-1}(\div\div,\Omega; \mathbb{S}\otimes\mathbb{R}^d), \\
			\label{intro-mix2}
			-\langle\div\div\boldsymbol{\Phi}, \boldsymbol{v}\rangle-(\boldsymbol{\sigma}(\boldsymbol{u}),\boldsymbol{\varepsilon}(\boldsymbol{v}))&= -(\boldsymbol{f},\boldsymbol{v}) \quad \forall \ \boldsymbol{v}\in H^1_0(\Omega;\mathbb{R}^d).
		\end{align}
\end{subequations}
Here $\widetilde{\mathcal{A}}$ is defined in \eqref{eq:mathcalA}. The boundary condition $\partial_n\boldsymbol{u}=0$ is weakly imposed in \eqref{intro-mix}.

The distributional viewpoint underlying the mixed formulation is inspired by the framework developed in \cite{HuangTang2025} for a scalar fourth-order elliptic singular perturbation problem. In the present elastic setting, the symmetric-stress structure, the physically meaningful $\mathbb{S}\otimes\mathbb{R}^d$-valued third-order double stress, and the requirement of estimates uniform in both $\iota$ and $\lambda$ necessitate a distinct finite element construction and analysis. Since
\[
H(\div,\Omega;\mathbb{S}\otimes\mathbb{R}^d)
\subset
H^{-1}(\div\div,\Omega;
\mathbb{S}\otimes\mathbb{R}^d),
\]
the distributional formulation permits an $H(\div)$-conforming approximation of the third-order double stress. Moreover, it reduces the regularity requirement on the displacement from $H_0^2(\Omega;\mathbb{R}^d)$ to $H_0^1(\Omega;\mathbb{R}^d)$, so that it suffices to discretize the short complex
\begin{equation}\label{complexdivpartH1}
	H_0^1(\Omega;\mathbb{R}^d)
	\xrightarrow{\div}
	L_0^2(\Omega)
	\xrightarrow{}0.
\end{equation}
More precisely, we approximate the double stress tensor $\boldsymbol{\Phi}$ by an $\mathbb{S}\otimes\mathbb{R}^d$-valued extension of the lowest-order Raviart--Thomas element \cite{RaviartThomas1977,Nedelec1980}, whose local shape function space is
\[
\Sigma(T;\mathbb{S}\otimes\mathbb{R}^d)
=
\mathbb{P}_0(T;\mathbb{S}\otimes\mathbb{R}^d)
+
\mathbb{P}_0(T;\mathbb{S})\otimes\boldsymbol{x},
\]
and approximate $\boldsymbol{u}$ by the vector-valued linear Crouzeix--Raviart element \cite{CR1973,BrennerSung1992}. Hence, the double stress and displacement are represented by lowest-order and linear finite element spaces, respectively, without the high-degree bubble enrichments used in the above constructions or a Nitsche-type treatment of the higher-order boundary condition.

The main contributions are threefold. First, we formulate the SGE model with the physically meaningful $\mathbb{S}\otimes\mathbb{R}^d$-valued third-order double stress as a primary mixed variable, thereby obtaining a distributional formulation that weakly incorporates the higher-order boundary condition. Second, this formulation enables a genuinely lowest-order RT--CR discretization in arbitrary dimensions without either high-degree bubble enrichment or a Nitsche-type boundary treatment. Third, we establish parameter-robust discrete stability and derive two complementary a priori error estimates: an optimal first-order estimate for fixed parameters and sufficiently smooth solutions, and, under the uniform regularity assumption stated below, a parameter-uniform $\mathcal{O}(\iota^{1/2}+h)$ estimate with constants independent of both $\iota$ and $\lambda$. In particular, the latter remains first order in $h$ in the boundary-layer regime $\iota^{1/2}\lesssim h$. We further construct a local quadratic post-processing for the displacement and develop a hybridized formulation. Numerical experiments in two and three dimensions support the theoretical findings.

The rest of this paper is organized as follows. Section~\ref{sec2} introduces some notation, derives the distributional mixed formulation, and recalls the uniform regularity results for the SGE model. In Section~\ref{sec3}, we develop and analyze the proposed mixed finite element method, including a local post-processing. Section~\ref{sec4} presents a hybridized formulation and its analysis. Numerical experiments in two and three dimensions are presented in Section~\ref{sec5}.

\section{Distributional mixed formulation of strain gradient elasticity}\label{sec2}
In this section, we introduce the third-order double stress $\boldsymbol{\Phi}:=\iota^2\grad\boldsymbol{\sigma}(\boldsymbol{u})$ and derive a distributional mixed formulation of the SGE model~\eqref{SGE0}.

\subsection{Tensor notation and differential operators}
Let $d\geq 2$, and let $\mathbb{S}$ denote the space of symmetric $d \times d$ matrices. The tensor product $\mathbb{S}\otimes\mathbb{R}^d$ is identified with the space of third-order tensors that are symmetric in the first two indices. For a second-order tensor $\boldsymbol{\tau}$ and a vector $\boldsymbol{v}$, we use the convention
\[
(\boldsymbol{\tau}\otimes\boldsymbol{v})_{ijk}:=\tau_{ij}v_k.
\]
For a third-order tensor $\boldsymbol{\Psi}=(\Psi_{ijk})$ and a vector $\boldsymbol{v}=(v_k)$, the tensor-vector contraction is defined by
\[
(\boldsymbol{\Psi}\boldsymbol{v})_{ij}:=(\boldsymbol{\Psi}\cdot\boldsymbol{v})_{ij}
:=\sum_{k=1}^d\Psi_{ijk}v_k.
\]
In particular, for a unit vector $\boldsymbol{n}$,
\[
(\boldsymbol{\Psi}\boldsymbol{n})_{ij}:=\sum_{k=1}^d\Psi_{ijk}n_k.
\]
All tensor inner products are Frobenius inner products.

For a second-order tensor $\boldsymbol{\tau}=(\tau_{ij})$, define
\[
\tr(\boldsymbol{\tau}):=\sum_{i=1}^d\tau_{ii},
\qquad
\dev\boldsymbol{\tau}:=\boldsymbol{\tau}-\frac{1}{d}\tr(\boldsymbol{\tau})\boldsymbol{I},
\]
where $\boldsymbol{I}$ is the $d\times d$ identity matrix. For a third-order tensor $\boldsymbol{\Psi}$, the trace is taken over the first two indices:
\[
\tr(\boldsymbol{\Psi})
:=\left(\sum_{i=1}^d\Psi_{ii1},\ldots,\sum_{i=1}^d\Psi_{iid}\right)^\intercal,
\qquad
\dev\boldsymbol{\Psi}
:=\boldsymbol{\Psi}-\frac{1}{d}\boldsymbol{I}\otimes\tr(\boldsymbol{\Psi}).
\]

Let $\mu>0$ and $\lambda>0$ be the Lam\'e constants. Define the fourth-order compliance tensor $\mathcal{A}:\mathbb{S}\to\mathbb{S}$ by
\[
\mathcal{A}\boldsymbol{\tau}
:=\frac{1}{2\mu}\left(\boldsymbol{\tau}-\frac{\lambda}{2\mu+d\lambda}\tr(\boldsymbol{\tau})\boldsymbol{I}\right)
=\frac{1}{2\mu}\dev\boldsymbol{\tau}
+\frac{1}{d(2\mu+d\lambda)}\tr(\boldsymbol{\tau})\boldsymbol{I},
\]
and its inverse $\mathcal{C}:\mathbb{S}\to\mathbb{S}$ by
\[
\mathcal{C}\boldsymbol{\tau}:=2\mu\boldsymbol{\tau}+\lambda\tr(\boldsymbol{\tau})\boldsymbol{I}.
\]
For $1\leq k\leq d$, let $\boldsymbol{\Psi}_{::k}:=(\Psi_{ijk})_{i,j=1}^d$. Define the extended compliance tensor $\widetilde{\mathcal{A}}=\mathcal{A}\otimes I_{\mathbb{R}^d}$ slice-wise by
\[
(\widetilde{\mathcal{A}}\boldsymbol{\Psi})_{::k}
:=\mathcal{A}(\boldsymbol{\Psi}_{::k}),
\qquad 1\leq k\leq d,
\]
where $I_{\mathbb R^d}:\mathbb R^d\to\mathbb R^d$ is the identity operator.
Equivalently,
\begin{equation}\label{eq:mathcalA}
\widetilde{\mathcal{A}}\boldsymbol{\Psi}
:=\frac{1}{2\mu}\left(\boldsymbol{\Psi}-\frac{\lambda}{2\mu+d\lambda}\boldsymbol{I}\otimes\tr(\boldsymbol{\Psi})\right)
=\frac{1}{2\mu}\dev\boldsymbol{\Psi}
+\frac{1}{d(2\mu+d\lambda)}\boldsymbol{I}\otimes\tr(\boldsymbol{\Psi}).
\end{equation}
Similarly, define $\widetilde{\mathcal{C}}=\mathcal{C}\otimes I_{\mathbb{R}^d}$ slice-wise by
\[
(\widetilde{\mathcal{C}}\boldsymbol{\Psi})_{::k}
:=\mathcal{C}(\boldsymbol{\Psi}_{::k}),
\qquad 1\leq k\leq d.
\]
Then $\widetilde{\mathcal{C}}$ is the inverse of $\widetilde{\mathcal{A}}$, and
\[
\widetilde{\mathcal{C}}\boldsymbol{\Psi}=2\mu\boldsymbol{\Psi}+\lambda\boldsymbol{I}\otimes\tr(\boldsymbol{\Psi}).
\]

\begin{lemma}
For a third-order tensor $\boldsymbol{\Psi}\in\mathbb S\otimes\mathbb R^d$, a second-order tensor $\boldsymbol{\tau}\in\mathbb S$ and a vector $\boldsymbol{v}\in\mathbb R^d$, we have
\begin{align}
\label{eq:trtensor3prop}
\tr(\boldsymbol{\Psi})\cdot\boldsymbol{v}&=\tr(\boldsymbol{\Psi}\boldsymbol{v}), \\
\label{eq:trtensorprop}
\tr(\boldsymbol{\tau}\otimes\boldsymbol{v})&=(\tr(\boldsymbol{\tau}))\boldsymbol{v}, \\
\label{eq:Atensorprop}
\widetilde{\mathcal{A}}(\boldsymbol{\tau}\otimes\boldsymbol{v})&=(\mathcal{A}\boldsymbol{\tau})\otimes\boldsymbol{v}, \\
\label{eq:Atildetensorprop}
(\widetilde{\mathcal{A}}\boldsymbol{\Psi})\boldsymbol{v}&=\mathcal{A}(\boldsymbol{\Psi}\boldsymbol{v}), \\
\label{eq:Ctensorprop}
\widetilde{\mathcal{C}}(\boldsymbol{\tau}\otimes\boldsymbol{v})&=(\mathcal{C}\boldsymbol{\tau})\otimes\boldsymbol{v}, \\
\label{eq:Ctildetensorprop}
(\widetilde{\mathcal{C}}\boldsymbol{\Psi})\boldsymbol{v}&=\mathcal{C}(\boldsymbol{\Psi}\boldsymbol{v}).
\end{align}
\end{lemma}
\begin{proof}
First, \eqref{eq:trtensor3prop}--\eqref{eq:trtensorprop} follow from the definition of the trace operator $\tr$.
By \eqref{eq:trtensorprop},
\begin{equation*}
\widetilde{\mathcal{A}}(\boldsymbol{\tau}\otimes\boldsymbol{v})=\frac{1}{2\mu} \left(\boldsymbol{\tau}\otimes\boldsymbol{v} - \frac{\lambda}{2\mu + d\lambda} (\tr(\boldsymbol{\tau})\boldsymbol{I} \otimes \boldsymbol{v}) \right),
\end{equation*}
which implies \eqref{eq:Atensorprop}. By the definition of $\widetilde{\mathcal{A}}\boldsymbol{\Psi}$ and \eqref{eq:trtensor3prop},
\begin{align*}
(\widetilde{\mathcal{A}}\boldsymbol{\Psi})\boldsymbol{v} = \frac{1}{2\mu} \left( \boldsymbol{\Psi}\boldsymbol{v} - \frac{\lambda}{2\mu + d\lambda} (\tr(\boldsymbol{\Psi})\cdot\boldsymbol{v})\boldsymbol{I} \right)= \frac{1}{2\mu} \left( \boldsymbol{\Psi}\boldsymbol{v} - \frac{\lambda}{2\mu + d\lambda} (\tr(\boldsymbol{\Psi}\boldsymbol{v}))\boldsymbol{I} \right).
\end{align*}
Thus, \eqref{eq:Atildetensorprop} follows. Similarly, we can derive \eqref{eq:Ctensorprop}--\eqref{eq:Ctildetensorprop} from \eqref{eq:trtensor3prop}--\eqref{eq:trtensorprop}.
\end{proof}

Given a vector-valued function $\boldsymbol{v}=(v_1,\ldots,v_d)^\intercal$, define
\[
\nabla\boldsymbol{v}:=\nabla\otimes\boldsymbol{v}=(\partial_i v_j)_{d\times d},
\qquad
\grad\boldsymbol{v}:=(\nabla\boldsymbol{v})^\intercal,
\qquad
\div\boldsymbol{v}:=\tr(\grad\boldsymbol{v}),
\]
and
\[
\boldsymbol{\varepsilon}(\boldsymbol{v})
:=\frac{1}{2}\left(\grad\boldsymbol{v}+(\grad\boldsymbol{v})^\intercal\right),
\qquad
\boldsymbol{\sigma}(\boldsymbol{v})
:=\mathcal{C}\boldsymbol{\varepsilon}(\boldsymbol{v})
=2\mu\boldsymbol{\varepsilon}(\boldsymbol{v})+\lambda(\div\boldsymbol{v})\boldsymbol{I}.
\]
Clearly, $\mathcal{A}\boldsymbol{\sigma}(\boldsymbol{v})=\boldsymbol{\varepsilon}(\boldsymbol{v})$. For a second-order tensor $\boldsymbol{\tau}=(\tau_{ij})$ and a third-order tensor $\boldsymbol{\Psi}=(\Psi_{ijk})$, define
\[
(\div\boldsymbol{\tau})_i:=\sum_{j=1}^d\partial_j\tau_{ij},
\qquad
(\grad\boldsymbol{\tau})_{ijk}:=\partial_k\tau_{ij},
\qquad
(\div\boldsymbol{\Psi})_{ij}:=\sum_{k=1}^d\partial_k\Psi_{ijk}.
\]

\subsection{Function spaces and mesh notation}
Let $\Omega\subset\mathbb{R}^d$ be a bounded polytope with boundary $\partial\Omega$. Given a bounded domain $D$, denote its diameter by $h_D$. For an integer $m\geq0$, denote by $H^m(D)$ the standard Sobolev space on $D$ with norm $\|\cdot\|_{m,D}$ and seminorm $|\cdot|_{m,D}$, and let $H_0^m(D)$ be the closure of $C_0^\infty(D)$ with respect to $\|\cdot\|_{m,D}$. The notation $(\cdot,\cdot)_D$ denotes the $L^2$ inner product on $D$. For $D=\Omega$, we abbreviate $\|\cdot\|_{m,D}$, $|\cdot|_{m,D}$, and $(\cdot,\cdot)_D$ as $\|\cdot\|_m$, $|\cdot|_m$, and $(\cdot,\cdot)$, respectively. The duality pairing between $H^{-1}(\Omega;\mathbb R^d)$ and $H_0^1(\Omega;\mathbb R^d)$ is denoted by $\langle\cdot,\cdot\rangle$.

Let $\mathcal{T}_h=\{T\}$ be a shape-regular simplicial mesh of $\Omega$. Set $h:=\max_{T\in\mathcal{T}_h}h_T$. Let $\mathcal{F}_h$ and $\mathring{\mathcal{F}}_h$ be the sets of all $(d-1)$-dimensional faces and interior faces, respectively. Let $T$ be a $d$-dimensional simplex with vertices $\{\texttt{v}_0,\ldots,\texttt{v}_d\}$. For $0\leq i\leq d$, let $F_i$ be the face opposite to $\texttt{v}_i$, and let $\lambda_i$ be the corresponding barycentric coordinate. Define $\mathcal{F}(T):=\{F_0,\ldots,F_d\}$. Let $\boldsymbol{n}_{\partial T}$ be the piecewise constant unit outward normal vector on $\partial T$. For each $F\in\mathcal{F}_h$, let $\boldsymbol{n}_F$ be a fixed unit normal vector, chosen as the outward unit normal when $F\subset\partial\Omega$. When no confusion arises, both $\boldsymbol{n}_{\partial T}$ and $\boldsymbol{n}_F$ are abbreviated as $\boldsymbol{n}$.

For two adjacent simplices $T^+$ and $T^-$ sharing an interior face $F$, define the jump of a function $v$ by
\[
[\![v]\!]|_F
:=(v|_{T^+})|_F\boldsymbol{n}_F\cdot\boldsymbol{n}_{\partial T^+}
+(v|_{T^-})|_F\boldsymbol{n}_F\cdot\boldsymbol{n}_{\partial T^-}.
\]
On a boundary face $F\subset\partial\Omega$, set $[\![v]\!]|_F:=v|_F$.

For a bounded domain $D\subset\mathbb{R}^d$ and an integer $k\geq0$, let $\mathbb{P}_k(D)$ be the space of polynomials on $D$ of total degree at most $k$, and set $\mathbb{P}_k(D):=\{0\}$ for $k<0$. For a linear space $V(D)$, define the broken spaces
\[
V(\mathcal{T}_h):=\prod_{T\in\mathcal{T}_h}V(T),
\qquad
V(\mathcal{F}_h):=\prod_{F\in\mathcal{F}_h}V(F).
\]
Denote $\mathbb{P}_k(\mathring{\mathcal{F}}_h):=\mathbb{P}_k(\mathcal{F}_h)\cap L^2(\mathring{\mathcal{F}}_h)$, where
\[
L^2(\mathring{\mathcal{F}}_h)
:=\{v\in L^2(\mathcal{F}_h):v|_F=0\quad\forall~F\in\mathcal{F}_h\setminus\mathring{\mathcal{F}}_h\}.
\]
For a finite-dimensional tensor space $\mathbb{X}\in\{\mathbb{R}^d,\mathbb{S},\mathbb{S}\otimes\mathbb{R}^d\}$, set
\[
V(D;\mathbb{X}):=V(D)\otimes\mathbb{X},
\quad
V(\mathcal{T}_h;\mathbb{X}):=\prod_{T\in\mathcal{T}_h}V(T;\mathbb{X}),
\quad
V(\mathcal{F}_h;\mathbb{X}):=\prod_{F\in\mathcal{F}_h}V(F;\mathbb{X}).
\]
The $L^2$-orthogonal projection onto $\mathbb{P}_k(D;\mathbb{X})$ is denoted by $Q_{k,D}$. Let $Q_{k,h}$ denote the $L^2$-orthogonal projection onto $\mathbb{P}_k(\mathcal{T}_h)$ or $\mathbb{P}_k(\mathcal{T}_h;\mathbb{X})$, and let $Q_{k,\mathcal{F}_h}$ denote the corresponding projection on $\mathcal{F}_h$.

We use $\grad_h$, $\div_h$, $\boldsymbol{\varepsilon}_h$, and $\boldsymbol{\sigma}_h$ for the elementwise versions of $\grad$, $\div$, $\boldsymbol{\varepsilon}$, and $\boldsymbol{\sigma}$, respectively. For a piecewise smooth scalar-, vector-, or tensor-valued function $v$, define
\[
\|v\|_{s,h}^2:=\sum_{T\in\mathcal{T}_h}\|v\|_{s,T}^2,
\qquad
|v|_{s,h}^2:=\sum_{T\in\mathcal{T}_h}|v|_{s,T}^2.
\]
For $\boldsymbol{v}\in H^1(\mathcal{T}_h;\mathbb{R}^d)$, introduce the discrete energy norm
\begin{align*}
{\interleave\boldsymbol{v}\interleave}_{1,h}^2
&:=(\boldsymbol{\sigma}_h(\boldsymbol{v}),\boldsymbol{\varepsilon}_h(\boldsymbol{v}))
+\sum_{F\in\mathcal{F}_h}h_F^{-1}\|[\![\boldsymbol{v}]\!]\|_{0,F}^2\\
&=2\mu\|\boldsymbol{\varepsilon}_h(\boldsymbol{v})\|_0^2
+\lambda\|\div_h\boldsymbol{v}\|_0^2
+\sum_{F\in\mathcal{F}_h}h_F^{-1}\|[\![\boldsymbol{v}]\!]\|_{0,F}^2.
\end{align*}
In particular, for $\boldsymbol{v}\in H_0^1(\Omega;\mathbb{R}^d)$,
\[
{\interleave\boldsymbol{v}\interleave}_{1,h}^2
={\interleave\boldsymbol{v}\interleave}_1^2
:=(\boldsymbol{\sigma}(\boldsymbol{v}),\boldsymbol{\varepsilon}(\boldsymbol{v}))
=2\mu\|\boldsymbol{\varepsilon}(\boldsymbol{v})\|_0^2
+\lambda\|\div\boldsymbol{v}\|_0^2.
\]
Throughout the paper, $\mu$ is fixed and positive. We use $a\lesssim b$ to mean $a\leq Cb$, where $C$ may depend on $\Omega$, $d$, and $\mu$ and, for discrete estimates, on the mesh shape regularity, but is independent of the mesh size $h$, the size parameter $\iota$, and the Lam\'e constant $\lambda$.
And $a\eqsim b$ means $a\lesssim b\lesssim a$. 

\subsection{Distributional mixed formulation}
By introducing the third-order double stress tensor $\boldsymbol{\Phi}:=\iota^2\grad\boldsymbol{\sigma}(\boldsymbol{u})$, it follows from the slice-wise definition of $\widetilde{\mathcal{A}}$ that
\begin{equation*}
\widetilde{\mathcal{A}}\boldsymbol{\Phi}= \iota^2\grad(\mathcal{A}\boldsymbol{\sigma}(\boldsymbol{u}))= \iota^2\grad\boldsymbol{\varepsilon}(\boldsymbol{u}).
\end{equation*}
Then we rewrite the SGE model \eqref{SGE0} as the following second-order system
\begin{equation}\label{SGE1}
	\begin{cases}
		\widetilde{\mathcal{A}}\boldsymbol{\Phi} = \iota^2\grad\boldsymbol{\varepsilon}(\boldsymbol{u}) &\mbox{in} \ \Omega,\\
		\div\div\boldsymbol{\Phi} - \div \boldsymbol{\sigma}(\boldsymbol{u}) = \boldsymbol{f} &\mbox{in} \ \Omega,\\
		\boldsymbol{u}=\partial_{n}\boldsymbol{u}=\boldsymbol{0} &\mbox{on} \ \partial\Omega.
	\end{cases}
\end{equation}

To present the distributional mixed formulation of the second-order system \eqref{SGE1}, introduce the Hilbert space
\[
H^{-1}(\div\div,\Omega;\mathbb{S}\otimes\mathbb{R}^d)
:=\{\boldsymbol{\Psi}\in L^2(\Omega;\mathbb{S}\otimes\mathbb{R}^d):\div\div\boldsymbol{\Psi}\in H^{-1}(\Omega;\mathbb{R}^d)\}.
\]
For any $\boldsymbol{g}\in H^{-1}(\Omega;\mathbb{R}^d)$, define
\[
\|\boldsymbol{g}\|_{-1,\mathcal{A}}
:=\sup_{\boldsymbol{0}\ne\boldsymbol{v}\in H_0^1(\Omega;\mathbb{R}^d)}
\frac{|\langle\boldsymbol{g},\boldsymbol{v}\rangle|}{{\interleave\boldsymbol{v}\interleave}_1}.
\]
By the first Korn inequality \cite[Corollary 11.2.22]{BrennerScott2008},
\[
\|\boldsymbol{g}\|_{-1,\mathcal A}\lesssim\|\boldsymbol{g}\|_0
\qquad\forall~\boldsymbol{g}\in L^2(\Omega;\mathbb R^d),
\]
with a constant independent of $\lambda$.
The parameter-weighted graph norm is
\[
\|\boldsymbol{\Psi}\|^2_{H_{\mathcal{A}}^{-1}(\iota\div\div)}
:=\iota^{-2}\|\boldsymbol{\Psi}\|_{\widetilde{\mathcal{A}}}^2
+\|\div\div\boldsymbol{\Psi}\|_{-1,\mathcal{A}}^2,
\]
where
\[
\|\boldsymbol{\Psi}\|_{\widetilde{\mathcal{A}}}^2
:=(\widetilde{\mathcal{A}}\boldsymbol{\Psi},\boldsymbol{\Psi})
=\frac{1}{2\mu}\|\dev\boldsymbol{\Psi}\|_0^2
+\frac{1}{d(2\mu+d\lambda)}\|\tr(\boldsymbol{\Psi})\|_0^2.
\]
%Note that $\|\boldsymbol{\Psi}\|_{\widetilde{\mathcal{A}}}\lesssim\|\boldsymbol{\Psi}\|_0$.

A distributional mixed formulation of the second-order system \eqref{SGE1} is to find $(\boldsymbol{\Phi},\boldsymbol{u})\in H^{-1}(\div\div,\Omega; \mathbb{S}\otimes\mathbb{R}^d)\times H^1_0(\Omega;\mathbb{R}^d)$ such that
\begin{subequations}\label{SGE-mix}
	\begin{align}
		\label{SGE-mix1}
		a(\boldsymbol{\Phi},\boldsymbol{\Psi})+b(\boldsymbol{\Psi},\boldsymbol{u}) &= 0 \quad\quad\quad\;\;\, \forall \ \boldsymbol{\Psi}\in  H^{-1}(\div\div,\Omega; \mathbb{S}\otimes\mathbb{R}^d), \\
		\label{SGE-mix2}
		b(\boldsymbol{\Phi},\boldsymbol{v})-c(\boldsymbol{u},\boldsymbol{v}) &= -(\boldsymbol{f},\boldsymbol{v}) \quad \forall \ \boldsymbol{v}\in H^1_0(\Omega;\mathbb{R}^d),
	\end{align}
\end{subequations}
where 
$$a(\boldsymbol{\Phi},\boldsymbol{\Psi}):=\iota^{-2}(\widetilde{\mathcal{A}}\boldsymbol{\Phi},\boldsymbol{\Psi}), \quad b(\boldsymbol{\Psi},\boldsymbol{v}):=-\langle\div\div\boldsymbol{\Psi}, \boldsymbol{v}\rangle, \quad c(\boldsymbol{u},\boldsymbol{v}):=(\boldsymbol{\sigma}(\boldsymbol{u}),\boldsymbol{\varepsilon}(\boldsymbol{v})).$$
\subsection{Well-posedness and equivalence}

\begin{theorem}
The distributional mixed formulation \eqref{SGE-mix} admits a unique solution $(\boldsymbol{\Phi},\boldsymbol{u})\in H^{-1}(\div\div,\Omega; \mathbb{S}\otimes\mathbb{R}^d)\times H^1_0(\Omega;\mathbb{R}^d)$. Moreover,
\begin{equation*}
\|\boldsymbol{\Phi}\|_{H_{\mathcal{A}}^{-1}(\iota\div\div)} + {\interleave\boldsymbol{u}\interleave}_{1} \lesssim \|\boldsymbol{f}\|_{-1,\mathcal{A}},
\end{equation*}
where the constant is independent of $\iota$ and $\lambda$.
\end{theorem}
\begin{proof}
% By Korn's inequality, ${\interleave\cdot\interleave}_1$ is a Hilbert norm on $H_0^1(\Omega;\mathbb R^d)$. For each fixed $\lambda$, it is equivalent to the standard $H^1$ norm, and hence $\|\cdot\|_{-1,\mathcal A}$ is equivalent to the standard $H^{-1}$ norm. Moreover, $\|\cdot\|_{\widetilde{\mathcal A}}$ is equivalent to the $L^2$ norm. Let $D:=\div\div$ denote the distributional operator with domain
% \[
% \mathcal D(D):=H^{-1}(\div\div,\Omega;\mathbb S\otimes\mathbb R^d).
% \]
% Then $D:\mathcal D(D)\subset L^2(\Omega;\mathbb S\otimes\mathbb R^d)\to H^{-1}(\Omega;\mathbb R^d)$ is a closed unbounded operator. Therefore, for each fixed $\iota$ and $\lambda$, the space $H^{-1}(\div\div,\Omega;\mathbb S\otimes\mathbb R^d)$ equipped with $\|\cdot\|_{H_{\mathcal A}^{-1}(\iota\div\div)}$ is a Hilbert space.

% The operator $\widetilde{\mathcal A}$ is self-adjoint and positive definite. 
First, the bilinear forms $a$, $b$, and $c$ are continuous in the parameter-weighted norms. 
% Indeed, the Cauchy--Schwarz inequality and the definition of $\|\cdot\|_{-1,\mathcal A}$ give
% \begin{align*}
% |a(\boldsymbol{\Phi},\boldsymbol{\Psi})|
% &\leq\|\boldsymbol{\Phi}\|_{H_{\mathcal A}^{-1}(\iota\div\div)}
% \|\boldsymbol{\Psi}\|_{H_{\mathcal A}^{-1}(\iota\div\div)},\\
% |b(\boldsymbol{\Psi},\boldsymbol{v})|
% &\leq\|\boldsymbol{\Psi}\|_{H_{\mathcal A}^{-1}(\iota\div\div)}
% {\interleave\boldsymbol{v}\interleave}_1,\\
% |c(\boldsymbol{u},\boldsymbol{v})|
% &\leq{\interleave\boldsymbol{u}\interleave}_1
% {\interleave\boldsymbol{v}\interleave}_1.
% \end{align*}
For any $\boldsymbol{\Psi}\in H^{-1}(\div\div,\Omega; \mathbb{S}\otimes\mathbb{R}^d)$, by the definition of the norm $\|\cdot\|_{H_{\mathcal{A}}^{-1}(\iota\div\div)}$, we have
\begin{equation}\label{eq:202506181}
a(\boldsymbol{\Psi}, \boldsymbol{\Psi})+ \sup_{\boldsymbol{0}\ne\boldsymbol{v}\in H^1_0(\Omega;\mathbb{R}^d)}\dfrac{b^2(\boldsymbol{\Psi}, \boldsymbol{v})}{{\interleave\boldsymbol{v}\interleave}^2_{1}} =\|\boldsymbol{\Psi}\|^2_{H_{\mathcal{A}}^{-1}(\iota\div\div)}.
\end{equation}
On the other hand, for any $\boldsymbol{\Psi}\in H^{-1}(\div\div,\Omega; \mathbb{S}\otimes\mathbb{R}^d)$ and $\boldsymbol{v}\in H^1_0(\Omega;\mathbb{R}^d)$, we have
\begin{equation*}
b^2(\boldsymbol{\Psi}, \boldsymbol{v})\leq \|\div\div\boldsymbol{\Psi}\|_{-1,\mathcal{A}}^2{\interleave\boldsymbol{v}\interleave}_{1}^2\leq \|\boldsymbol{\Psi}\|_{H_{\mathcal{A}}^{-1}(\iota\div\div)}^2{\interleave\boldsymbol{v}\interleave}_{1}^2,
\end{equation*}
which implies
\begin{equation*}
\sup_{\boldsymbol{0}\ne\boldsymbol{\Psi}\in H^{-1}(\div\div,\Omega; \mathbb{S}\otimes\mathbb{R}^d)}\frac{b^2(\boldsymbol{\Psi}, \boldsymbol{v})}{\|\boldsymbol{\Psi}\|_{H_{\mathcal{A}}^{-1}(\iota\div\div)}^2} \leq  {\interleave\boldsymbol{v}\interleave}^2_{1}\quad\forall~\boldsymbol{v}\in H^1_0(\Omega;\mathbb{R}^d).
\end{equation*}
Since $c(\boldsymbol{v}, \boldsymbol{v}) = \interleave \boldsymbol{v} \interleave^2_{1}$, it follows that
\begin{equation*}
c(\boldsymbol{v}, \boldsymbol{v}) + \sup_{\boldsymbol{0}\ne\boldsymbol{\Psi}\in H^{-1}(\div\div,\Omega; \mathbb{S}\otimes\mathbb{R}^d)}\frac{b^2(\boldsymbol{\Psi}, \boldsymbol{v})}{\|\boldsymbol{\Psi}\|_{H_{\mathcal{A}}^{-1}(\iota\div\div)}^2} \eqsim {\interleave\boldsymbol{v}\interleave}^2_{1}\quad\forall~\boldsymbol{v}\in H^1_0(\Omega;\mathbb{R}^d).
\end{equation*}
Thus, \eqref{eq:202506181} and the last equivalence are precisely conditions (2.12)--(2.13) of \cite[Theorem~2.6]{zulehner2011}. Zulehner's theorem yields the well-posedness of \eqref{SGE-mix} and the stated parameter-robust estimate.
% Finally, by combining the last equivalence and \eqref{eq:202506181}, we apply the Zulehner theory \cite{zulehner2011} to get the well-posedness of the mixed formulation \eqref{SGE-mix}. 
\end{proof}

Since
\[
H(\div,\Omega;\mathbb{S}\otimes\mathbb{R}^d)
\subset H^{-1}(\div\div,\Omega;\mathbb{S}\otimes\mathbb{R}^d),
\]
we have
\begin{equation}\label{eq:Hdiv-embedding}
-\langle\div\div\boldsymbol{\Psi},\boldsymbol{v}\rangle
=(\div\boldsymbol{\Psi},\boldsymbol{\varepsilon}(\boldsymbol{v}))
\quad
\forall~\boldsymbol{\Psi}\in H(\div,\Omega;\mathbb{S}\otimes\mathbb{R}^d),
\ \boldsymbol{v}\in H_0^1(\Omega;\mathbb{R}^d),
\end{equation}
where
\[
H(\div,\Omega;\mathbb{S}\otimes\mathbb{R}^d)
:=\{\boldsymbol{\Psi}\in L^2(\Omega;\mathbb{S}\otimes\mathbb{R}^d):\div\boldsymbol{\Psi}\in L^2(\Omega;\mathbb{S})\}.
\]
Indeed,
\[
|\langle\div\div\boldsymbol{\Psi},\boldsymbol{v}\rangle|
\leq\|\div\boldsymbol{\Psi}\|_0\|\boldsymbol{\varepsilon}(\boldsymbol{v})\|_0
\lesssim\|\div\boldsymbol{\Psi}\|_0{\interleave\boldsymbol{v}\interleave}_1,
\]
which proves the continuous embedding. We will discretize $\boldsymbol{\Phi}$ using a finite element subspace of this $H(\div)$ space.

\begin{remark}\rm
Since $\mathcal{A}$ and $\mathcal{C}$ have constant coefficients and act only on the first two tensor indices,
\[
\div(\widetilde{\mathcal{A}}\boldsymbol{\Psi})
=\mathcal{A}(\div\boldsymbol{\Psi}),
\qquad
\div(\widetilde{\mathcal{C}}\boldsymbol{\Psi})
=\mathcal{C}(\div\boldsymbol{\Psi}).
\]
Consequently, $\widetilde{\mathcal{A}}$ and $\widetilde{\mathcal{C}}$ map
$H(\div,\Omega;\mathbb{S}\otimes\mathbb{R}^d)$ into itself. Since
$\widetilde{\mathcal{C}}=\widetilde{\mathcal{A}}^{-1}$, they are inverse isomorphisms of
$H(\div,\Omega;\mathbb{S}\otimes\mathbb{R}^d)$.
\end{remark}

We interpret problem \eqref{SGE0} in the following weak sense: find $\boldsymbol{u}\in H_0^2(\Omega;\mathbb{R}^d)$ such that
\begin{equation}\label{SGE-weak}
\iota^2(\grad\boldsymbol{\sigma}(\boldsymbol{u}),\grad\boldsymbol{\varepsilon}(\boldsymbol{v}))
+(\boldsymbol{\sigma}(\boldsymbol{u}),\boldsymbol{\varepsilon}(\boldsymbol{v}))
=(\boldsymbol{f},\boldsymbol{v})
\quad\forall~\boldsymbol{v}\in H_0^2(\Omega;\mathbb{R}^d).
\end{equation}

\begin{theorem}
A pair $(\boldsymbol{\Phi},\boldsymbol{u})$ solves the distributional mixed formulation \eqref{SGE-mix} if and only if $\boldsymbol{u}\in H_0^2(\Omega;\mathbb{R}^d)$ solves \eqref{SGE-weak} and
\[
\boldsymbol{\Phi}=\iota^2\grad\boldsymbol{\sigma}(\boldsymbol{u}).
\]
Consequently, the distributional mixed formulation \eqref{SGE-mix} is equivalent to the weak formulation \eqref{SGE-weak} of the SGE problem~\eqref{SGE0}.
\end{theorem}
\begin{proof}
Suppose first that $\boldsymbol{u}\in H_0^2(\Omega;\mathbb{R}^d)$ solves \eqref{SGE-weak}, and set $\boldsymbol{\Phi}:=\iota^2\grad\boldsymbol{\sigma}(\boldsymbol{u})$. Testing \eqref{SGE-weak} with compactly supported smooth functions shows in the sense of distributions that
\[
\div\div\boldsymbol{\Phi}-\div\boldsymbol{\sigma}(\boldsymbol{u})=\boldsymbol{f}.
\]
Since $\boldsymbol{u}\in H^2(\Omega;\mathbb{R}^d)$ and $\boldsymbol f\in L^2(\Omega;\mathbb R^d)$, we have
$\boldsymbol f+\div\boldsymbol{\sigma}(\boldsymbol{u})\in L^2(\Omega;\mathbb R^d)\subset H^{-1}(\Omega;\mathbb{R}^d)$. Hence
$\boldsymbol{\Phi}\in H^{-1}(\div\div,\Omega;\mathbb{S}\otimes\mathbb{R}^d)$. Since both $\div\div\boldsymbol{\Phi}-\div\boldsymbol{\sigma}(\boldsymbol{u})$ and $\boldsymbol f$ belong to $H^{-1}(\Omega;\mathbb R^d)$, the distributional identity extends by density to every test function in $H_0^1(\Omega;\mathbb R^d)$, and \eqref{SGE-mix2} follows. Moreover,
\[
\widetilde{\mathcal A}\boldsymbol{\Phi}
=\iota^2\grad\boldsymbol{\varepsilon}(\boldsymbol{u}).
\]
For every $\boldsymbol{\Psi}\in H^{-1}(\div\div,\Omega;\mathbb{S}\otimes\mathbb{R}^d)$, choose $\boldsymbol{u}_m\in C_0^\infty(\Omega;\mathbb R^d)$ such that $\boldsymbol{u}_m\to\boldsymbol{u}$ in $H^2(\Omega;\mathbb R^d)$. Then
\begin{align*}
\langle\div\div\boldsymbol{\Psi},\boldsymbol{u}\rangle
&=\lim_{m\to\infty}\langle\div\div\boldsymbol{\Psi},\boldsymbol{u}_m\rangle\\
&=\lim_{m\to\infty}(\boldsymbol{\Psi},\grad\boldsymbol{\varepsilon}(\boldsymbol{u}_m))
=(\boldsymbol{\Psi},\grad\boldsymbol{\varepsilon}(\boldsymbol{u})),
\end{align*}
and therefore \eqref{SGE-mix1} holds.

Conversely, let $(\boldsymbol{\Phi},\boldsymbol{u})$ solve \eqref{SGE-mix}. Taking $\boldsymbol{\Psi}\in C_0^\infty(\Omega;\mathbb{S}\otimes\mathbb{R}^d)$ in \eqref{SGE-mix1} yields
\begin{equation}\label{eq:grad-epsilon-recovery}
\grad\boldsymbol{\varepsilon}(\boldsymbol{u})
=\iota^{-2}\widetilde{\mathcal A}\boldsymbol{\Phi}
\quad\text{in }L^2(\Omega;\mathbb{S}\otimes\mathbb{R}^d).
\end{equation}
Thus, $\boldsymbol{\varepsilon}(\boldsymbol{u})\in H^1(\Omega;\mathbb{S})$. The identity
\[
\partial_j\partial_k u_i
=\partial_k\varepsilon_{ij}(\boldsymbol{u})
+\partial_j\varepsilon_{ik}(\boldsymbol{u})
-\partial_i\varepsilon_{jk}(\boldsymbol{u})
\]
holds in the sense of distributions and implies $\boldsymbol{u}\in H^2(\Omega;\mathbb{R}^d)$.

Let $\gamma$ denote the usual trace operator. The componentwise normal trace operator
\[
\gamma_n:H(\div,\Omega;\mathbb S\otimes\mathbb R^d)
\longrightarrow H^{-1/2}(\partial\Omega;\mathbb S),
\qquad
\gamma_n\boldsymbol{\Psi}=\boldsymbol{\Psi}\boldsymbol{n},
\]
is continuous and surjective by the standard normal-trace theorem for $H(\div)$; see \cite[Chapter~I, Section~2]{GiraultRaviart1986}. Here $\langle\cdot,\cdot\rangle_{\partial\Omega}$ denotes the duality pairing
\[
H^{-1/2}(\partial\Omega;\mathbb S)
\times H^{1/2}(\partial\Omega;\mathbb S).
\]
Since $\boldsymbol{\varepsilon}(\boldsymbol{u})\in H^1(\Omega;\mathbb S)$, its trace belongs to $H^{1/2}(\partial\Omega;\mathbb S)$. For any $\boldsymbol{\Psi}\in H(\div,\Omega;\mathbb{S}\otimes\mathbb{R}^d)$, equations \eqref{eq:Hdiv-embedding}, \eqref{SGE-mix1}, and \eqref{eq:grad-epsilon-recovery}, followed by Green's formula, imply
\[
0=(\grad\boldsymbol{\varepsilon}(\boldsymbol{u}),\boldsymbol{\Psi})
+(\boldsymbol{\varepsilon}(\boldsymbol{u}),\div\boldsymbol{\Psi})
=\langle\gamma_n\boldsymbol{\Psi},\gamma(\boldsymbol{\varepsilon}(\boldsymbol{u}))\rangle_{\partial\Omega}.
\]
Surjectivity therefore gives $\gamma(\boldsymbol{\varepsilon}(\boldsymbol{u}))=\boldsymbol{0}$ on $\partial\Omega$. Moreover, $\boldsymbol{u}\in H^2(\Omega;\mathbb R^d)$ and $\gamma\boldsymbol{u}=\boldsymbol{0}$. 
% Hence, on every open boundary face, the tangential derivative of its trace vanishes:
% \[
% \gamma(\partial_{\boldsymbol t}\boldsymbol{u})
% =\partial_{\boldsymbol t}(\gamma\boldsymbol{u})=\boldsymbol{0}.
% \]
% For every tangential vector $\boldsymbol{t}$,
% \[
% \boldsymbol{t}\cdot\gamma(\partial_n\boldsymbol{u})
% =2\boldsymbol{t}^{\intercal}\gamma(\boldsymbol{\varepsilon}(\boldsymbol{u}))\boldsymbol{n}
% -\boldsymbol{n}\cdot\gamma(\partial_{\boldsymbol{t}}\boldsymbol{u})=0,
% \qquad
% \boldsymbol{n}\cdot\gamma(\partial_n\boldsymbol{u})
% =\boldsymbol{n}^{\intercal}\gamma(\boldsymbol{\varepsilon}(\boldsymbol{u}))\boldsymbol{n}=0.
% \]
% Therefore, $\gamma(\partial_n\boldsymbol{u})=\boldsymbol{0}$ and thus $\gamma(\grad\boldsymbol{u})=\boldsymbol{0}$ on $\partial\Omega$. Applying the standard scalar trace characterization componentwise on the bounded polytopal domain $\Omega$ gives
% \[
% H_0^2(\Omega;\mathbb R^d)
% =\{\boldsymbol v\in H^2(\Omega;\mathbb R^d):
% \gamma\boldsymbol v=\boldsymbol 0,\ \gamma(\grad\boldsymbol v)=\boldsymbol 0\}.
% \]
We conclude that $\boldsymbol{u}\in H_0^2(\Omega;\mathbb{R}^d)$.

Finally, the slice-wise definition of $\widetilde{\mathcal A}$ gives
$\widetilde{\mathcal A}\grad\boldsymbol{\sigma}(\boldsymbol{u})
=\grad\boldsymbol{\varepsilon}(\boldsymbol{u})$. Since $\widetilde{\mathcal A}$ is invertible, \eqref{eq:grad-epsilon-recovery} yields
$\boldsymbol{\Phi}=\iota^2\grad\boldsymbol{\sigma}(\boldsymbol{u})$. For every $\boldsymbol{v}\in H_0^2(\Omega;\mathbb R^d)$, equation \eqref{SGE-mix2} gives
\begin{align*}
(\boldsymbol f,\boldsymbol v)
&=\langle\div\div\boldsymbol\Phi,\boldsymbol v\rangle
+(\boldsymbol\sigma(\boldsymbol u),\boldsymbol\varepsilon(\boldsymbol v))\\
&=(\boldsymbol\Phi,\grad\boldsymbol\varepsilon(\boldsymbol v))
+(\boldsymbol\sigma(\boldsymbol u),\boldsymbol\varepsilon(\boldsymbol v))\\
&=\iota^2(\grad\boldsymbol\sigma(\boldsymbol u),\grad\boldsymbol\varepsilon(\boldsymbol v))
+(\boldsymbol\sigma(\boldsymbol u),\boldsymbol\varepsilon(\boldsymbol v)).
\end{align*}
Thus, $\boldsymbol u$ solves \eqref{SGE-weak}.
\end{proof}
%\begin{remark}\rm
%	In contrast to the Hellinger--Reissner formulation of linear elasticity, one
%	cannot obtain a dev-div type estimate that controls the full $L^2$ norm of a
%	third-order tensor in the present setting. Indeed, let
%	$\boldsymbol q\in C_0^\infty(\Omega;\mathbb R^d)$ be a nonzero divergence-free
%	vector field and set
%	\[
%	\boldsymbol\Psi:=\boldsymbol I\otimes\boldsymbol q,
%	\qquad
%	\Psi_{ijk}=\delta_{ij}q_k .
%	\]
%	Then
%	\[
%	\dev\boldsymbol\Psi=0,
%	\qquad
%	\div\boldsymbol\Psi=(\div\boldsymbol q)\boldsymbol I=0,
%	\qquad
%	\div\div\boldsymbol\Psi=0,
%	\]
%	while $\boldsymbol\Psi\neq0$. Hence the full $L^2$ norm of
%	$\boldsymbol\Psi$ cannot be controlled by $\dev\boldsymbol\Psi$ and
%	$\div\div\boldsymbol\Psi$. Therefore, the natural norm for the tensor variable
%	in this formulation is the compliance-weighted norm
%	$\|\cdot\|_{\widetilde{\mathcal A}}$, rather than the $L^2$ norm $\|\cdot\|_0$.
%\end{remark}

\begin{remark}\rm
	By the definition of the compliance operator
	$\widetilde{\mathcal A}$, we have
	\[
	\frac{1}{2\mu+d\lambda}\|\boldsymbol{\Psi}\|_0^2
	\leq
	\|\boldsymbol{\Psi}\|_{\widetilde{\mathcal A}}^2
	\leq
	\frac{1}{2\mu}\|\boldsymbol{\Psi}\|_0^2.
	\]
	Thus, for each fixed $\lambda$,
	$\|\cdot\|_{\widetilde{\mathcal A}}$ is equivalent to the $L^2$ norm $\|\cdot\|_0$. However, this norm equivalence is not uniform with respect to $\lambda$, since the coefficient $(2\mu+d\lambda)^{-1}$ in the lower bound tends to zero as $\lambda\to\infty$. Moreover, adding the $\div\div$ term does not restore uniform control of the $L^2$ norm.
	
	To see this, let
	$\boldsymbol q\in C_0^\infty(\Omega;\mathbb R^d)$ be a nonzero
	divergence-free vector field and set
	\[
	\boldsymbol{\Psi}:=\boldsymbol I\otimes\boldsymbol q.
	\]
	Then
	\[
	\dev\boldsymbol{\Psi}=0,
	\qquad
	\div\boldsymbol{\Psi}
	=(\div\boldsymbol q)\boldsymbol I=0,
	\qquad
	\div\div\boldsymbol{\Psi}=0,
	\]
	while $\|\boldsymbol{\Psi}\|_0>0$. For each fixed $\iota$, $\|\boldsymbol{\Psi}\|_{H_{\mathcal A}^{-1}(\iota\div\div)}=\iota^{-1}\sqrt{\frac{d}{2\mu+d\lambda}}\|\boldsymbol{q}\|_0$ tends to zero as $\lambda\to\infty$. Thus, the parameter-weighted graph norm does not uniformly control the unweighted $L^2$ norm as $\lambda\to\infty$. Consequently, the stability estimate above cannot be converted, merely by norm equivalence, into a $\lambda$-uniform $L^2$ estimate for the tensor variable. This motivates measuring the tensor variable in the compliance-weighted norm.
\end{remark}

\subsection{Uniform regularity assumption}
The reduced elasticity problem associated with the limit $\iota\to0$ is obtained by formally setting $\iota=0$ in \eqref{SGE0}. Its weak formulation is: find $\boldsymbol{u}_0\in H_0^1(\Omega;\mathbb R^d)$ such that
\begin{equation}\label{SGElinear-weak}
(\boldsymbol{\sigma}(\boldsymbol{u}_0),\boldsymbol{\varepsilon}(\boldsymbol{v}))
=(\boldsymbol{f},\boldsymbol{v})
\qquad\forall~\boldsymbol{v}\in H_0^1(\Omega;\mathbb R^d).
\end{equation}
For sufficiently regular $\boldsymbol{u}_0$, this weak problem corresponds to the strong form
\begin{equation}\label{SGElinear}
\begin{cases}
-2\mu\div(\boldsymbol{\varepsilon}(\boldsymbol{u}_0))-\lambda \nabla\div\boldsymbol{u}_0=\boldsymbol{f} &\mbox{in} \ \Omega,\\
\boldsymbol{u}_0=0 &\mbox{on} \ \partial\Omega.
\end{cases}
\end{equation}
% \begin{equation}\label{SGElinear}
% 	\begin{cases}
% 		\mathcal{A}\boldsymbol{\sigma} = \boldsymbol{\varepsilon}(\boldsymbol{u})&\mbox{in} \ \Omega,\\
% 		-\div\boldsymbol{\sigma}=\boldsymbol{f} &\mbox{in} \ \Omega,\\
% 		\boldsymbol{u}=\boldsymbol{0} &\mbox{on} \ \partial\Omega,
% 	\end{cases}
% \end{equation}
\begin{assumption}[Uniform regularity]\label{ass:uniform-regularity}
Let $\boldsymbol{u}$ be the weak solution of \eqref{SGE-weak}, equivalently of \eqref{SGE0}, and let $\boldsymbol{u}_0\in H_0^1(\Omega;\mathbb R^d)$ solve \eqref{SGElinear-weak}. We assume that
\begin{equation}\label{elasregularity}
\|\boldsymbol{u}_0\|_2+\lambda\|\div\boldsymbol{u}_0\|_{1}\lesssim\|\boldsymbol f\|_{0},
\end{equation}
and
\begin{align}
	\label{Regularity-u}
	|\boldsymbol{u}-\boldsymbol{u}_0|_1+\iota\|\boldsymbol{u}\|_2+\iota^2\|\boldsymbol{u}\|_3&\lesssim\iota^{1/2}\|\boldsymbol{f}\|_0,\\
	\label{Regularity-divu}
	\lambda\|\div(\boldsymbol{u}-\boldsymbol{u}_0)\|_0+\lambda\iota|\div\boldsymbol{u}|_1+\lambda\iota^2\|\div\boldsymbol{u}\|_2&\lesssim\iota^{1/2}\|\boldsymbol{f}\|_0.
\end{align}
The constants are independent of $\iota\in(0,1]$ and $\lambda$.
\end{assumption}

The finite element construction and the abstract error analysis below are valid in arbitrary dimensions $d\geq2$, provided Assumption~\ref{ass:uniform-regularity} holds. The currently available verification of this assumption is summarized next.

For convex domains in two and three dimensions, estimate \eqref{elasregularity} follows from the regularity theory in \cite{BrennerSung1992,MR2641539}. The estimates \eqref{Regularity-u}--\eqref{Regularity-divu} were established in \cite{ChenHuangHuang2023} under Assumptions~3.1--3.2 therein. In that reference, Assumption~3.1 is verified in two dimensions for convex polygons, whereas Assumption~3.2 is known for smooth boundaries and requires a separate operator-pencil analysis on polytopal domains.

%%%%%%%%%%%%%%%%%%%%%%%%%%%%%%%%%%%%%%%%
%% FEM
%%%%%%%%%%%%%%%%%%%%%%%%%%%%%%%%%%%%%%%%
\section{A third-order tensor mixed finite element method with linear elements}\label{sec3}
In this section, we develop and analyze a mixed finite element method for the SGE model~\eqref{SGE1}. The method combines an $H(\div)$-conforming finite element space for the third-order double stress tensor with the linear Crouzeix--Raviart space for the displacement. We establish parameter-robust discrete stability and derive both optimal-order and parameter-uniform a priori error estimates. A local post-processing is also presented.

% \subsection{Finite elements and interpolation operators}
\subsection{Finite element spaces}

For third-order tensors, the divergence operator acts exclusively on the last index. This structural property facilitates a natural extension of vector-valued $H(\div)$-conforming finite elements to third-order tensor-valued $H(\div)$-conforming finite elements.
Based on this insight, we adopt the lowest-order $H(\div)$-conforming finite elements for third-order tensors to discretize the double stress tensor $\boldsymbol{\Phi}$.
We take the local shape function space to be
$$
\Sigma(T;\mathbb{S}\otimes\mathbb{R}^d):=\mathbb{P}_0(T;\mathbb{S}\otimes\mathbb{R}^d)+\mathbb{P}_{0}(T;\mathbb{S})\otimes\boldsymbol{x}.
$$
The degrees of freedom (DoFs) are given by
\begin{equation}
\label{tensor-dof}
	(\boldsymbol{\Psi}\boldsymbol{n}, \boldsymbol{q})_F,  \quad\boldsymbol{q}\in \mathbb{P}_0(F;\mathbb{S}), \, F\in\mathcal{F}(T).
\end{equation}
This space--DoF pair is the analogue, for third-order tensors, of the classical Raviart--Thomas (RT) element~\cite{RaviartThomas1977,Nedelec1980}; thus, the space $\Sigma(T;\mathbb{S}\otimes\mathbb{R}^d)$ is uniquely determined by the DoFs in \eqref{tensor-dof}.
%Following the tangential-normal decomposition in~\cite{ChenHuang2024Hdiv}, the face DoFs in \eqref{tensor-dof} can be decomposed as follows.
%Let
%$F\in\mathcal F(T)$ be a $(d-1)$-dimensional face, and let
%$\{\boldsymbol t_i^F\}_{i=1}^{d-1}$ be a basis of the tangent space of $F$.
%Define
%\begin{align*}
%	\mathscr T^F(\mathbb S)
%	&:=
%	\operatorname{span}
%	\bigl\{
%	\operatorname{sym}(\boldsymbol t_i^F\otimes \boldsymbol t_j^F):
%	1\leq i\leq j\leq d-1
%	\bigr\},\\
%	\mathscr N^F(\mathbb S)
%	&:=
%	\operatorname{span}
%	\bigl\{
%	\operatorname{sym}(\boldsymbol t_i^F\otimes \boldsymbol n_F):
%	1\leq i\leq d-1
%	\bigr\}
%	\oplus
%	\operatorname{span}\{\boldsymbol n_F\otimes \boldsymbol n_F\}.
%\end{align*}
%Then
%\[
%\mathbb S
%=
%\mathscr T^F(\mathbb S)
%\oplus
%\mathscr N^F(\mathbb S).
%\]
%Accordingly, the DoFs in \eqref{tensor-dof} are split into
%\begin{align}
%	\label{tensor-dof-t}
%	(\boldsymbol{\Psi}\boldsymbol{n}, \boldsymbol{q})_F,
%	&\quad
%	\boldsymbol{q}\in \mathbb{P}_0(F;\mathscr T^F(\mathbb S)),
%	\, F\in\mathcal{F}(T),\\
%	\label{tensor-dof-n}
%	(\boldsymbol{\Psi}\boldsymbol{n}, \boldsymbol{q})_F,
%	&\quad
%	\boldsymbol{q}\in \mathbb{P}_0(F;\mathscr N^F(\mathbb S)),
%	\, F\in\mathcal{F}(T).
%\end{align}

Given a basis $\{\boldsymbol{\varsigma}_j, 1\leq j\leq d(d+1)/2\}$ of $\mathbb{S}$, a convenient basis of $\Sigma(T;\mathbb{S}\otimes\mathbb{R}^d)$ associated with the faces is
\begin{equation}
\label{tensor-basis}
\{\boldsymbol{\varsigma}_j\otimes(\texttt{v}_i-\boldsymbol{x})(\boldsymbol{n}_{F_i}\cdot\nabla\lambda_i), \quad 0\leq i\leq d, \; 1\leq j\leq d(d+1)/2\}.
\end{equation}
%If, for the face $F_i$, a basis of $\mathbb S$ and its dual basis are chosen
%according to
%\[
%\mathbb S
%=
%\mathscr T^{F_i}(\mathbb S)
%\oplus
%\mathscr N^{F_i}(\mathbb S),
%\]
%then the basis functions in \eqref{tensor-basis} are correspondingly dual to
%the decomposed DoFs \eqref{tensor-dof-t}--\eqref{tensor-dof-n}.

The global finite element space for tensors is defined as
\begin{align*}
	\Sigma_{h}:=\{\boldsymbol{\Psi}\in L^{2}(\Omega;\mathbb{S}\otimes\mathbb{R}^d)&:\boldsymbol{\Psi}|_T\in \Sigma(T;\mathbb{S}\otimes\mathbb{R}^d)~\textrm{for}~\textrm{each}~T\in\mathcal{T}_h, \\
	&\qquad\qquad \textrm{the~DoFs in \eqref{tensor-dof} are single-valued}\}. 
\end{align*}
Clearly, $\Sigma_{h}\subset H(\div,\Omega;\mathbb{S}\otimes\mathbb{R}^d)$.

%\begin{remark}\rm
%	The third-order tensor finite element space can be naturally modified for the homogeneous
%	clamped--simply supported boundary conditions discussed in Remark~\ref{remark-mixbd}.
%	Let
%	\[
%	\mathcal F_h^{\rm ss}
%	:=
%	\{F\in\mathcal F_h^\partial:\ F\subset\Gamma_{\rm ss}\}.
%	\]
%	We define
%	\[
%	\Sigma_{h,{\rm ss}}
%	:=
%	\{
%	\boldsymbol\Psi_h\in\Sigma_h:
%	\textrm{the~DoF~}\eqref{tensor-dof-n}~\textrm{vanishes~on~each}~F\in\mathcal F_h^{\rm ss}
%	\}.
%	\]
%	Hence
%	\[
%	\Sigma_{h,{\rm ss}}
%	\subset
%	H_{\rm ss}(\div,\Omega;\mathbb S\otimes\mathbb R^d).
%	\]
%	For the homogeneous clamped--simply supported boundary condition, one uses
%	$\Sigma_{h,{\rm ss}}$ as the tensor finite element space in the mixed method
%	below. The stability and error analysis in this paper is restricted to the
%	homogeneous fully clamped case.
%\end{remark}

To discretize the displacement field $\boldsymbol{u}$, we employ the vector-valued nonconforming linear element introduced in~\cite{CR1973,Brenner2015}. The local shape function space is given by $\mathbb{P}_1(T;\mathbb{R}^d)$, and the associated DoFs are defined by
\begin{equation}
	\label{CR-dof}
	(\boldsymbol{v},\boldsymbol{q})_F, \quad\boldsymbol{q}\in \mathbb{P}_0(F;\mathbb{R}^d), \, F\in\mathcal{F}(T).
\end{equation}
The global vector-valued nonconforming linear element space is defined as
\begin{align*}
	V_{h} &:= \{\boldsymbol{v}\in L^2(\Omega;\mathbb{R}^d):\boldsymbol{v}|_T\in\mathbb{P}_1(T;\mathbb{R}^d) \textrm{ for each }T\in\mathcal{T}_h; \textrm{ DoF \eqref{CR-dof} is } \\
	&\qquad\qquad\;\;\;\;\textrm{ single-valued across each face in $\mathring{\mathcal{F}}_h$, and vanishes on $\partial\Omega$}\}.
\end{align*}
Functions in $V_{h}$ satisfy the weak continuity condition
\begin{equation}\label{weak-c}
	([\![\boldsymbol{v}]\!],\boldsymbol{q})_F = 0 \quad \forall \ \boldsymbol{v}\in V_{h}, \boldsymbol{q}\in\mathbb{P}_{0}(F;\mathbb{R}^d), F\in\mathcal{F}_h.
\end{equation}
The weak continuity \eqref{weak-c} yields the discrete Poincar\'{e} inequality \cite{MR1974504}
\begin{equation}\label{CRpoincare}
\|\boldsymbol{v}\|_0^2 + \sum_{F\in\mathcal F_h}h^{-1}_F\|[\![\boldsymbol{v}]\!]\|^2_{0,F}\lesssim |\boldsymbol{v}|^2_{1,h}, \quad \forall \ \boldsymbol{v}\in V_{h}.
\end{equation}

\begin{lemma}
	The following norm equivalence holds
	\begin{equation}
		\label{CRnormeq}
		|\boldsymbol{v}|^2_{1,h}\eqsim\|\boldsymbol{\varepsilon}_h(\boldsymbol{v})\|^2_0+\sum_{F\in\mathcal F_h}h^{-1}_F\|[\![\boldsymbol{v}]\!]\|^2_{0,F}, \quad \forall \ \boldsymbol{v}\in V_{h}.
	\end{equation}
\end{lemma}
\begin{proof}
By the broken Korn inequality (see \cite[(1.22)]{MR2047078} and \cite[(34)]{ArnoldBrezziMarini2005}), we have
\begin{equation*}
\|\boldsymbol{v}\|^2_{1,h}\lesssim\|\boldsymbol{\varepsilon}_h(\boldsymbol{v})\|^2_0+\sum_{F\in\mathcal F_h}h^{-1}_F\|[\![\boldsymbol{v}]\!]\|^2_{0,F}, \quad \forall \ \boldsymbol{v}\in V_{h}.
\end{equation*}
Combining this with the discrete Poincar\'{e} inequality \eqref{CRpoincare} yields the norm equivalence \eqref{CRnormeq}.
\end{proof}

For $\boldsymbol{v}\in H_0^2(\Omega;\mathbb{R}^d)+V_h$, define the discrete parameter-dependent quantity
$$
{\interleave \boldsymbol{v}\interleave}^2_{\iota,h}:=\iota^2{\interleave\boldsymbol{v}\interleave}^2_{2,h}+{\interleave\boldsymbol{v}\interleave}^2_{1,h},
$$
where the discrete $H^2$ seminorm is defined by
\begin{align*}
{\interleave \boldsymbol{v}\interleave}_{2,h}^2&:=\sum_{F\in\mathcal{F}_h}h_F^{-1}([\![\boldsymbol{\sigma}_h(\boldsymbol{v})]\!],[\![\boldsymbol{\varepsilon}_h(\boldsymbol{v})]\!])_F \\
&\;=\sum_{F\in\mathcal{F}_h}h_F^{-1}(2\mu\|[\![\boldsymbol{\varepsilon}_h(\boldsymbol{v})]\!]\|^2_{0,F}+\lambda\|[\![\div_h\boldsymbol{v}]\!]\|^2_{0,F}).
\end{align*}
The parameter-dependent quantity ${\interleave \cdot\interleave}_{\iota,h}$ is a norm on $V_h$.

\subsection{Interpolation operators}

%Recall the Brezzi--Douglas--Marini (BDM) element \cite{ChenHuang2022div,ChenChenHuangWei2024,ChenHuang2024,BrezziDouglasMarini1985,BrezziDouglasDuranFortin1987,nedelec1986new} and the Raviart--Thomas (RT) element \cite{RaviartThomas1977,Nedelec1980}.

An $L^2$-bounded projection operator onto the third-order tensor-valued finite element space $\Sigma_{h}$ is essential for the robust analysis below. To this end,
let $I^{\div}_{h}: L^2(\Omega;\mathbb{S}\otimes\mathbb{R}^d)\rightarrow\Sigma_{h}$ denote the third-order tensor-valued counterpart of the $L^2$-bounded projection operator devised in \cite{ArnoldGuzman2021}.
We have
\begin{align}
	\label{IhdivProp1}
	\qquad\quad\;\, \div(I^{\div}_{h}\boldsymbol{\Psi}) &= Q_{0,h}\div\boldsymbol{\Psi}\quad\;\forall~\boldsymbol{\Psi}\in H(\div,\Omega;\mathbb{S}\otimes\mathbb{R}^d),\\
	\label{IhdivProp2}
\|\boldsymbol{\Psi}-I^{\div}_{h}\boldsymbol{\Psi}\|_{0}&\lesssim h^{s}|\boldsymbol{\Psi}|_{s}\qquad\;\;\;\,\forall~\boldsymbol{\Psi}\in H^{s}(\Omega;\mathbb{S}\otimes\mathbb{R}^d), \ 0\leq s\leq 1.
\end{align}
In particular, $\|I_h^{\div}\boldsymbol{\Psi}\|_0\lesssim\|\boldsymbol{\Psi}\|_0$ for every $\boldsymbol{\Psi}\in L^2(\Omega;\mathbb{S}\otimes\mathbb{R}^d)$.

Let $I_h^{\rm CR}:H_0^1(\Omega;\mathbb{R}^d)\rightarrow V_{h}$ be the canonical interpolation operator associated with the DoF~\eqref{CR-dof}. 
The following properties hold (cf.~\cite{BrennerSung1992,Brenner2015}):
\begin{equation}
\label{IhCRProp1}
\grad_h(I_h^{\rm CR} \boldsymbol{v}) = Q_{0,h}(\grad\boldsymbol{v}) \quad \forall~\boldsymbol{v} \in H_0^1(\Omega; \mathbb{R}^d),
\end{equation}
and for any $T \in \mathcal{T}_h$ and $\boldsymbol{v}\in H^1_0(\Omega;\mathbb{R}^d)\cap H^{s}(\Omega;\mathbb{R}^d)$ with $1 \leq s \leq 2$,
\begin{equation}
\label{IhCRProp2}
\|\boldsymbol{v} - I_h^{\rm CR} \boldsymbol{v} \|_{0,T}+h_T^{1/2}\|\boldsymbol{v}-I^{\rm CR}_h\boldsymbol{v}\|_{0,\partial T} + h_T |\boldsymbol{v} - I_h^{\rm CR} \boldsymbol{v}|_{1,T} \lesssim h_T^s |\boldsymbol{v}|_{s,T},
\end{equation}
\begin{equation}
\label{IhCRProp3}
\|\div(\boldsymbol{v} - I_h^{\rm CR} \boldsymbol{v})\|_{0,T} \lesssim h_T^{s-1} |\div\boldsymbol{v}|_{s-1,T}.
\end{equation}

\begin{lemma}
\label{interpolation-I-lemma2}
Let $\boldsymbol{u}\in H_0^2(\Omega;\mathbb{R}^d)$ be the solution of the SGE problem \eqref{SGE0}.
Assume that the regularity estimates \eqref{elasregularity}--\eqref{Regularity-divu} hold.
%  and the solution $\boldsymbol{u}_0\in H^1_0(\Omega;\mathbb{R}^d)$ of problem \eqref{SGElinear} satisfy the regularity \eqref{elasregularity}. Assume $\boldsymbol{u}\in H^{2}(\Omega;\mathbb{R}^d)$, 
We have
	\begin{align}
		\label{IhCRerror1}
		\interleave\boldsymbol{u}-I^{\rm CR}_h\boldsymbol{u}\interleave_{1,h} & \lesssim  h(|\boldsymbol{u}|_2+\sqrt{\lambda}|\div\boldsymbol{u}|_1), \\ %(\iota h+h^2)|\boldsymbol{\sigma}(\boldsymbol{u})|_2
		\label{IhCRerror2}
		\interleave\boldsymbol{u}-I^{\rm CR}_h\boldsymbol{u}\interleave_{\iota,h} & \lesssim(\iota^{1/2}+h)\|\boldsymbol{f}\|_0.
	\end{align}
\end{lemma}
\begin{proof}
The estimate \eqref{IhCRerror1} immediately follows from estimates \eqref{IhCRProp2}--\eqref{IhCRProp3} of $I^{\rm CR}_h$.
% By the trace inequality and \eqref{IhCRProp2},
% \begin{equation}\label{IhCRProp3}
% \|\boldsymbol{u}-I^{\rm CR}_h\boldsymbol{u}\|_{0,\partial T}\lesssim h_T^{-1/2}\|\boldsymbol{u}-I^{\rm CR}_h\boldsymbol{u}\|_{0,T} + h_T^{1/2}|\boldsymbol{u}-I^{\rm CR}_h\boldsymbol{u}|_{1,T}\lesssim h_T^{3/2-j}|\boldsymbol{u}|_{2-j,T},
% \end{equation}
% with $j=0,1$.
Applying estimates \eqref{IhCRProp2}--\eqref{IhCRProp3} again, we obtain
	\begin{align*}
		\interleave\boldsymbol{u}-I^{\rm CR}_h\boldsymbol{u}\interleave_{1,h}&\leq
		\interleave(\boldsymbol{u}-\boldsymbol{u}_0)-I^{\rm CR}_h(\boldsymbol{u}-\boldsymbol{u}_0)\interleave_{1,h} +{\interleave\boldsymbol{u}_0-I^{\rm CR}_h\boldsymbol{u}_0\interleave}_{1,h} \\
		&\lesssim \|\boldsymbol{\sigma}(\boldsymbol{u}-\boldsymbol{u}_0)\|_{0}+h\|\boldsymbol{\sigma}(\boldsymbol{u}_0)\|_1.
	\end{align*}
	Hence,  we get from the last inequality and \eqref{elasregularity}--\eqref{Regularity-divu} that
\begin{equation*}
\interleave\boldsymbol{u}-I^{\rm CR}_h\boldsymbol{u}\interleave_{1,h} \lesssim(\iota^{1/2}+h)\|\boldsymbol{f}\|_0.
\end{equation*}
By \eqref{IhCRProp1}, the error estimate of $Q_{0,h}$, and the regularity estimates \eqref{Regularity-u}--\eqref{Regularity-divu},
\begin{align*}
&\quad\; \iota^2{\interleave \boldsymbol{u}-I^{\rm CR}_h\boldsymbol{u}\interleave}_{2,h}^2 \\
&=\iota^2\sum_{F\in\mathcal{F}_h}h_F^{-1}(2\mu\|[\![\boldsymbol{\varepsilon}_h(\boldsymbol{u}-I^{\rm CR}_h\boldsymbol{u})]\!]\|^2_{0,F}+\lambda\|[\![\div_h(\boldsymbol{u}-I^{\rm CR}_h\boldsymbol{u})]\!]\|^2_{0,F}) \\
&=\iota^2\sum_{F\in\mathcal{F}_h}h_F^{-1}(2\mu\|[\![\boldsymbol{\varepsilon}(\boldsymbol{u})-Q_{0,h}(\boldsymbol{\varepsilon}(\boldsymbol{u}))]\!]\|^2_{0,F}+\lambda\|[\![\div\boldsymbol{u}-Q_{0,h}(\div\boldsymbol{u})]\!]\|^2_{0,F}) \\
&\lesssim \iota^2(2\mu|\boldsymbol{\varepsilon}(\boldsymbol{u})|_1^2 + \lambda|\div\boldsymbol{u}|_1^2)\lesssim \iota\|\boldsymbol{f}\|_0^2.
\end{align*}
Therefore, \eqref{IhCRerror2} holds from the combination of the last two estimates.
\end{proof}
%Then we have the following inf-sup condition.
%\begin{lemma}
%	It holds the discrete inf-sup condition
%	\begin{equation}
%		\label{CRinfsup}
%		\|q\|_0\lesssim \sup_{\boldsymbol{v}\in V_h}\dfrac{(\div_h\boldsymbol{v},q)}{|\boldsymbol{v}|_{1,h}}\quad\forall~ q_h\in\mathring{\mathbb{P}}_{0}(\mathcal{T}_h).
%	\end{equation}
%\end{lemma}
%\begin{proof}
%	For $q_h\in\mathring{\mathbb{P}}_{0}(\mathcal{T}_h)$, according to $\div H_0^1(\Omega;\mathbb{R}^d)=L_0^2(\Omega)$ \cite{CostabelMcIntosh2010}, there exists $\boldsymbol{v}\in H_0^1(\Omega;\mathbb{R}^d)$  satisfying 
%	\begin{equation*}
%		\div\boldsymbol{v} = q_h, \quad \textrm{ and }\quad |\boldsymbol{v}|_1\lesssim\|q_h\|_0.
%	\end{equation*}
%	Take $\boldsymbol{v}_h = I^{\rm CR}_h\boldsymbol{v}$. By \eqref{IhCRProp1} and \eqref{IhCRProp2}, we have \begin{equation*}
%		\div_h\boldsymbol{v}_h = q_h, \quad \textrm{ and }\quad |\boldsymbol{v}_h|_{1,h}\lesssim|\boldsymbol{v}|_1\lesssim\|q_h\|_0,
%	\end{equation*}
%	which implies \eqref{CRinfsup}.
%\end{proof}
\subsection{Mixed tensor finite element method} 
Based on the distributional mixed formulation \eqref{SGE-mix}, we propose the following mixed tensor finite element method for the SGE model \eqref{SGE0}:
find $(\boldsymbol{\Phi}_h, \boldsymbol{u}_h)\in\Sigma_{h}\times V_{h}$ such that
\begin{subequations}\label{SGE-MFEM}
	\begin{align}
		\label{SGE-MFEM1}
		a(\boldsymbol{\Phi}_h,\boldsymbol{\Psi}_h)+b_h(\boldsymbol{\Psi}_h,\boldsymbol{u}_h) &= 0 \qquad\qquad\; \forall \ \boldsymbol{\Psi}_h\in  \Sigma_{h}, \\
		\label{SGE-MFEM2}
		b_h(\boldsymbol{\Phi}_h,\boldsymbol{v}_h)-c_h(\boldsymbol{u}_h,\boldsymbol{v}_h)&= -(\boldsymbol{f}, \boldsymbol{v}_h) \quad \forall \ \boldsymbol{v}_h\in V_{h},
	\end{align}
\end{subequations} 
where the discrete bilinear forms are defined as follows:
\begin{align*}
	 b_h(\boldsymbol{\Psi}_h,\boldsymbol{v}_h)&:=(\div\boldsymbol{\Psi}_h, \boldsymbol{\varepsilon}_h(\boldsymbol{v}_h)), \\	c_h(\boldsymbol{u}_h,\boldsymbol{v}_h)&:=(\boldsymbol{\sigma}_h(\boldsymbol{u}_h),\boldsymbol{\varepsilon}_h(\boldsymbol{v}_h))+\sum_{F\in\mathcal F_h}h^{-1}_F([\![\boldsymbol{u}_h]\!],[\![\boldsymbol{v}_h]\!])_F.
\end{align*}
Recall that $a(\boldsymbol{\Phi}_h,\boldsymbol{\Psi}_h):=\iota^{-2}(\widetilde{\mathcal{A}}\boldsymbol{\Phi}_h,\boldsymbol{\Psi}_h)$.
%Here, the stabilization parameter $\eta > 0$ is an arbitrary fixed constant independent of the mesh size $h$, the size parameter $\iota$, and the Lam\'e coefficient $\lambda$. Unlike the penalty parameter in interior penalty methods \cite{BabuskaZlamal1973,Wheeler1978,Arnold1979,Arnold1982}, $\eta$ is not required to be sufficiently large.
Clearly,
\begin{equation}\label{eq:chcoercivity}
c_h(\boldsymbol{v}_h, \boldsymbol{v}_h) = {\interleave \boldsymbol{v}_h\interleave}^2_{1,h}\qquad \forall \ \boldsymbol{v}_h\in V_{h}.
\end{equation}

\begin{lemma}\label{inf-supH2-lemma}
The following discrete inf-sup condition holds:
\begin{equation}
\label{inf-supH2}
{\interleave \boldsymbol{v}_h\interleave}_{2,h}\lesssim\sup_{\boldsymbol{\Psi}_h\in \Sigma_{h}}\frac{(\div\boldsymbol{\Psi}_h, \boldsymbol{\varepsilon}_h(\boldsymbol{v}_h))}{\|\boldsymbol{\Psi}_h\|_{\widetilde{\mathcal{A}}}} \quad \forall \ \boldsymbol{v}_h\in V_{h}.
\end{equation}
\end{lemma}
\begin{proof}
Thanks to DoF \eqref{tensor-dof},
let $\boldsymbol{\Psi}_h\in\Sigma_{h}$ be determined by 
\begin{equation*}
(\boldsymbol{\Psi}_h\boldsymbol{n}_F)|_F=h_F^{-1}[\![\boldsymbol{\sigma}_h(\boldsymbol{v}_h)]\!]|_F \quad\quad \forall \ F\in\mathcal{F}_h.
\end{equation*}
By the explicit expression of the basis functions \eqref{tensor-basis}, we have for each $T\in\mathcal{T}_h$ that
\begin{equation*}
\boldsymbol{\Psi}_h|_T = \sum_{i=0}^dh_{F_i}^{-1}[\![\boldsymbol{\sigma}_h(\boldsymbol{v}_h)]\!]|_{F_i} \otimes(\texttt{v}_i-\boldsymbol{x})(\boldsymbol{n}_{F_i}\cdot\nabla\lambda_i).
\end{equation*}
By \eqref{eq:Atensorprop},
\begin{equation*}
(\widetilde{\mathcal{A}}\boldsymbol{\Psi}_h)|_T = \sum_{i=0}^dh_{F_i}^{-1}[\![\boldsymbol{\varepsilon}_h(\boldsymbol{v}_h)]\!]|_{F_i} \otimes(\texttt{v}_i-\boldsymbol{x})(\boldsymbol{n}_{F_i}\cdot\nabla\lambda_i).
\end{equation*}
Noting that $|(\texttt{v}_i-\boldsymbol{x})(\boldsymbol{n}_{F_i}\cdot\nabla\lambda_i)|\lesssim 1$,
we obtain
\begin{align*}
(\widetilde{\mathcal{A}}\boldsymbol{\Psi}_h,\boldsymbol{\Psi}_h)_T&\lesssim \sum_{i,j=0}^d(h_{F_i}h_{F_j})^{-1}\big|([\![\boldsymbol{\sigma}_h(\boldsymbol{v}_h)]\!]|_{F_i}, [\![\boldsymbol{\varepsilon}_h(\boldsymbol{v}_h)]\!]|_{F_j})_T\big| \\
&\lesssim \sum_{i=0}^dh_{F_i}^{-2}([\![\boldsymbol{\sigma}_h(\boldsymbol{v}_h)]\!]|_{F_i}, [\![\boldsymbol{\varepsilon}_h(\boldsymbol{v}_h)]\!]|_{F_i})_T \\
&\lesssim \sum_{i=0}^dh_{F_i}^{-1}([\![\boldsymbol{\sigma}_h(\boldsymbol{v}_h)]\!], [\![\boldsymbol{\varepsilon}_h(\boldsymbol{v}_h)]\!])_{F_i},
\end{align*}
where the second inequality follows from the Cauchy--Schwarz inequality in the positive-definite $\mathcal C$-inner product, the inequality $(\sum_{i=0}^d a_i)^2\leq(d+1)\sum_{i=0}^d a_i^2$, and the identity $\boldsymbol{\sigma}_h(\boldsymbol{v}_h)=\mathcal C\boldsymbol{\varepsilon}_h(\boldsymbol{v}_h)$. The last inequality uses the scaling $|T|\eqsim h_{F_i}|F_i|$.
This implies
\begin{equation}
\label{inf-supH2-pf2}
\|\boldsymbol{\Psi}_h\|^2_{\widetilde{\mathcal{A}}}=(\widetilde{\mathcal{A}}\boldsymbol{\Psi}_h,\boldsymbol{\Psi}_h)\lesssim\sum_{F\in\mathcal{F}_h}h_F^{-1}([\![\boldsymbol{\sigma}_h(\boldsymbol{v}_h)]\!],[\![\boldsymbol{\varepsilon}_h(\boldsymbol{v}_h)]\!])_F={\interleave \boldsymbol{v}_h\interleave}^2_{2,h}.
\end{equation}
Applying integration by parts, we get
\begin{align*}
(\div\boldsymbol{\Psi}_h, \boldsymbol{\varepsilon}_h(\boldsymbol{v}_h))=\sum_{F\in\mathcal{F}_h}(\boldsymbol{\Psi}_h\boldsymbol{n}_F,[\![\boldsymbol{\varepsilon}_h(\boldsymbol{v}_h)]\!])_F={\interleave \boldsymbol{v}_h\interleave}_{2,h}^2,
\end{align*}
which together with \eqref{inf-supH2-pf2} yields \eqref{inf-supH2}.
\end{proof}

\begin{theorem}\label{wp-theorem}
The following discrete stability estimate holds:
	\begin{align}
		\label{discretestability}
		&\iota^{-1}\|\boldsymbol{\Phi}_h\|_{\widetilde{\mathcal{A}}}+{\interleave \boldsymbol{u}_h\interleave}_{\iota,h} \\
		\notag
		&\quad \lesssim\sup_{\boldsymbol{\Psi}_h\in \Sigma_{h}, \boldsymbol{v}_h\in V_{h}}
		\frac{a(\boldsymbol{\Phi}_h, \boldsymbol{\Psi}_h)+b_h(\boldsymbol{\Psi}_h, \boldsymbol{u}_h)+b_h(\boldsymbol{\Phi}_h, \boldsymbol{v}_h)-c_h(\boldsymbol{u}_h, \boldsymbol{v}_h)}{\iota^{-1}\|\boldsymbol{\Psi}_h\|_{\widetilde{\mathcal{A}}}+{\interleave \boldsymbol{v}_h\interleave}_{\iota,h}}
	\end{align}
	for any $\boldsymbol{\Phi}_h\in\Sigma_{h}$ and $\boldsymbol{u}_h\in V_{h}$.
Then the mixed tensor finite element method \eqref{SGE-MFEM} is well-posed.
\end{theorem}
\begin{proof}
By \eqref{eq:Atildetensorprop} and the inverse inequality, for any $\boldsymbol{\Psi}_h\in \Sigma_{h}$, we have
\begin{align*}
\sum_{F\in\mathcal{F}_h}h_F(\mathcal{A}(\boldsymbol{\Psi}_h\boldsymbol{n}), \boldsymbol{\Psi}_h\boldsymbol{n})_F&=\sum_{F\in\mathcal{F}_h}h_F((\widetilde{\mathcal{A}}\boldsymbol{\Psi}_h)\boldsymbol{n}, \boldsymbol{\Psi}_h\boldsymbol{n})_F \\
& \lesssim \sum_{T\in\mathcal{T}_h}h_T(\widetilde{\mathcal{A}}\boldsymbol{\Psi}_h, \boldsymbol{\Psi}_h)_{\partial T} \lesssim (\widetilde{\mathcal{A}}\boldsymbol{\Psi}_h, \boldsymbol{\Psi}_h).
\end{align*}
Applying integration by parts, for any $\boldsymbol{\Psi}_h\in \Sigma_{h}$ and $\boldsymbol{v}_h\in V_{h}$, we obtain
\begin{align*}
b_h(\boldsymbol{\Psi}_h, \boldsymbol{v}_h) &= \sum_{F\in\mathcal{F}_h}(\boldsymbol{\Psi}_h\boldsymbol{n}_F, [\![\boldsymbol{\varepsilon}_h(\boldsymbol{v}_h)]\!])_F \leq {\interleave \boldsymbol{v}_h\interleave}_{2,h}\bigg(\sum_{F\in\mathcal{F}_h}h_F(\mathcal{A}(\boldsymbol{\Psi}_h\boldsymbol{n}), \boldsymbol{\Psi}_h\boldsymbol{n})_F\bigg)^{1/2} \\
& \lesssim \|\boldsymbol{\Psi}_h\|_{\widetilde{\mathcal{A}}}{\interleave \boldsymbol{v}_h\interleave}_{2,h}\leq \iota^{-1}\|\boldsymbol{\Psi}_h\|_{\widetilde{\mathcal{A}}}{\interleave \boldsymbol{v}_h\interleave}_{\iota,h}.
\end{align*}
This, together with the identity $a(\boldsymbol{\Psi}_h,\boldsymbol{\Psi}_h)=\iota^{-2}\|\boldsymbol{\Psi}_h\|^2_{\widetilde{\mathcal{A}}}$, implies
\begin{equation*}
\iota^{-2}\|\boldsymbol{\Psi}_h\|^2_{\widetilde{\mathcal{A}}}\leq a(\boldsymbol{\Psi}_h, \boldsymbol{\Psi}_h)+ \sup_{\boldsymbol{v}_h\in V_{h}}\frac{b_h^2(\boldsymbol{\Psi}_h, \boldsymbol{v}_h)}{{\interleave \boldsymbol{v}_h\interleave}_{\iota,h}^2} \lesssim\iota^{-2}\|\boldsymbol{\Psi}_h\|^2_{\widetilde{\mathcal{A}}} \qquad \forall \ \boldsymbol{\Psi}_h\in 	\Sigma_h.
\end{equation*}
On the other hand, by the discrete inf-sup condition \eqref{inf-supH2}, we have
\begin{align*}
{\interleave \boldsymbol{v}_h\interleave}^2_{\iota,h}\lesssim c_h(\boldsymbol{v}_h, \boldsymbol{v}_h) + \sup_{\boldsymbol{\Psi}_h\in \Sigma_{h}}\frac{b_h^2(\boldsymbol{\Psi}_h, \boldsymbol{v}_h)}{\iota^{-2}\|\boldsymbol{\Psi}_h\|_{\widetilde{\mathcal{A}}}^2} \lesssim{\interleave \boldsymbol{v}_h\interleave}^2_{\iota,h} \qquad \forall \ \boldsymbol{v}_h\in V_{h}.
\end{align*}
Combining the two norm equivalences above, we apply Zulehner's theory to conclude the discrete stability \eqref{discretestability} and the well-posedness of the mixed finite element method~\eqref{SGE-MFEM}.
\end{proof}

\subsection{Error analysis}

For later use, the identity $\mathcal A\boldsymbol{\sigma}(\boldsymbol{u})=\boldsymbol{\varepsilon}(\boldsymbol{u})$, the formula
\[
\div\boldsymbol{u}=\frac{\tr\boldsymbol{\sigma}(\boldsymbol{u})}{2\mu+d\lambda},
\]
and the distributional identity for the second derivatives of $\boldsymbol{u}$ imply, for $m=1,2$,
\begin{equation}
\label{stress-controls-u}
|\boldsymbol{u}|_{m+1}+\sqrt{\lambda}\,|\div\boldsymbol{u}|_m\lesssim |\boldsymbol{\sigma}(\boldsymbol{u})|_m.
\end{equation}
The constant in \eqref{stress-controls-u} is independent of $\lambda$.

\begin{lemma}
Let $(\boldsymbol{\Phi},\boldsymbol{u})\in H^{-1}(\div\div,\Omega; \mathbb{S}\otimes\mathbb{R}^d)\times H_0^1(\Omega;\mathbb{R}^d)$ be the solution of problem~\eqref{SGE-mix}. If $\boldsymbol{\Phi}\in H^2(\Omega;\mathbb{S}\otimes\mathbb{R}^d)$ and $\boldsymbol{u}\in H^{2}(\Omega;\mathbb{R}^d)$, then the following consistency estimate holds:
\begin{equation}
\label{consistencyerr-1}
b_h(\boldsymbol{\Phi},\boldsymbol{v}_h)-c_h(\boldsymbol{u},\boldsymbol{v}_h)+(\boldsymbol{f}, \boldsymbol{v}_h)\lesssim h(|\div\boldsymbol{\Phi}|_{1}+|\boldsymbol{\sigma}(\boldsymbol{u})|_{1})|\boldsymbol{v}_h|_{1,h}\quad \forall~\boldsymbol{v}_h\in V_{h}.
\end{equation}
\end{lemma}
\begin{proof}
Set $\boldsymbol{w}=\div\boldsymbol{\Phi}-\boldsymbol{\sigma}(\boldsymbol{u})$ for simplicity. Then the second equation of problem~\eqref{SGE1} becomes $\div\boldsymbol{w}=\boldsymbol{f}$.

For any $\boldsymbol{v}_h\in V_h$, integration by parts and the weak continuity \eqref{weak-c} yield
\begin{align*}
&\quad\; b_h(\boldsymbol{\Phi},\boldsymbol{v}_h)-c_h(\boldsymbol{u},\boldsymbol{v}_h)+(\boldsymbol{f},\boldsymbol{v}_h) \\
&= \sum_{T\in \mathcal{T}_h}\sum_{F\in\mathcal{F}(T)}(\boldsymbol{w}\boldsymbol{n},\boldsymbol{v}_h)_{F} = \sum_{F\in\mathcal{F}_h}(\boldsymbol{w}\boldsymbol{n}-Q_{0,F}(\boldsymbol{w}\boldsymbol{n}), [\![\boldsymbol{v}_h]\!])_{F}.
\end{align*}
From the error estimate of $Q_{0,F}$, the trace inequality and \eqref{CRnormeq}, we get
\begin{align*}
b_h(\boldsymbol{\Phi},\boldsymbol{v}_h)-c_h(\boldsymbol{u},\boldsymbol{v}_h)+(\boldsymbol{f},\boldsymbol{v}_h)\lesssim h|\boldsymbol{w}|_1\bigg(\sum_{F\in \mathcal{F}_h}h_F^{-1}\|[\![\boldsymbol{v}_h]\!]\|_{0,F}^2\bigg)^{1/2}\lesssim h|\boldsymbol{w}|_1|\boldsymbol{v}_h|_{1,h},
\end{align*}
which gives \eqref{consistencyerr-1}.
\end{proof}

\begin{theorem}\label{thm:errresult1}
Let $(\boldsymbol{\Phi},\boldsymbol{u})\in H^{-1}(\div\div,\Omega; \mathbb{S}\otimes\mathbb{R}^d)\times H_0^1(\Omega;\mathbb{R}^d)$ and $(\boldsymbol{\Phi}_h,\boldsymbol{u}_h)\in \Sigma_{h}\times V_{h}$ be the solution of problem \eqref{SGE-mix} and the mixed FEM \eqref{SGE-MFEM}, respectively. If $\boldsymbol{\Phi}\in H^2(\Omega;\mathbb{S}\otimes\mathbb{R}^d)$ and $\boldsymbol{u}\in H^3(\Omega;\mathbb{R}^d)$, then
\begin{align}
\label{errresult1}
&	\iota^{-1}\|\boldsymbol{\Phi}-\boldsymbol{\Phi}_h\|_{\widetilde{\mathcal{A}}}+\iota{\interleave I_h^{\rm CR}\boldsymbol{u}-\boldsymbol{u}_h\interleave}_{2,h}+{\interleave \boldsymbol{u}-\boldsymbol{u}_h\interleave}_{1,h} \\
\notag
&\qquad\qquad\qquad\qquad\qquad\qquad\lesssim h(|\div\boldsymbol{\Phi}|_{1}+\iota|\boldsymbol{\sigma}(\boldsymbol{u})|_{2}+|\boldsymbol{\sigma}(\boldsymbol{u})|_{1}).
\end{align}
\end{theorem}
\begin{proof}
Take $\boldsymbol{\Psi}_h\in\Sigma_{h}$ and $\boldsymbol{v}_h\in V_{h}$.
From the equation \eqref{SGE-mix1} and the mixed method \eqref{SGE-MFEM}, we have the error equations
\begin{align}
\label{eq:erreqn1}
a(I_{h}^{\div}\boldsymbol{\Phi}-\boldsymbol{\Phi}_h,\boldsymbol{\Psi}_h)+b_h(\boldsymbol{\Psi}_h,\boldsymbol{u}-\boldsymbol{u}_h)&=a(I_{h}^{\div}\boldsymbol{\Phi}-\boldsymbol{\Phi},\boldsymbol{\Psi}_h), \\
\label{eq:erreqn2}
b_h(\boldsymbol{\Phi}-\boldsymbol{\Phi}_h,\boldsymbol{v}_h)-c_h(\boldsymbol{u}-\boldsymbol{u}_h,\boldsymbol{v}_h) &= b_h(\boldsymbol{\Phi},\boldsymbol{v}_h)-c_h(\boldsymbol{u},\boldsymbol{v}_h)+(\boldsymbol{f}, \boldsymbol{v}_h).
\end{align}

The projection properties \eqref{IhdivProp1} and \eqref{IhCRProp1} give, for every $\boldsymbol{\Psi}_h\in\Sigma_h$ and $\boldsymbol{v}_h\in V_h$,
\[
b_h(\boldsymbol{\Psi}_h,I^{\rm CR}_h\boldsymbol{u}-\boldsymbol{u})=0,\qquad
b_h(I_h^{\div}\boldsymbol{\Phi}-\boldsymbol{\Phi},\boldsymbol{v}_h)=0,
\]
and
\[
(\boldsymbol{\sigma}_h(I^{\rm CR}_h\boldsymbol{u}-\boldsymbol{u}),\boldsymbol{\varepsilon}_h(\boldsymbol{v}_h))=0.
\]
Thus, the jump term displayed below is the only correction in
$c_h(I^{\rm CR}_h\boldsymbol{u}-\boldsymbol{u},\boldsymbol{v}_h)$.
Since $\boldsymbol{\Phi}=\iota^2\grad\boldsymbol{\sigma}(\boldsymbol{u})$, the estimate \eqref{IhdivProp2} gives
\begin{equation}\label{errresult-pf20260630}
	\|I_{h}^{\div}\boldsymbol{\Phi}-\boldsymbol{\Phi}\|_{\widetilde{\mathcal{A}}}\lesssim\|I_{h}^{\div}\boldsymbol{\Phi}-\boldsymbol{\Phi}\|_0\lesssim h|\boldsymbol{\Phi}|_1\lesssim \iota^{2}h|\boldsymbol{\sigma}(\boldsymbol{u})|_{2}.
\end{equation}
Then the error equation~\eqref{eq:erreqn1} becomes
%\begin{align}\label{errresult-pf1}
%&\quad a(I_{h}^{\div}\boldsymbol{\Phi}-\boldsymbol{\Phi}_h,\boldsymbol{\Psi}_h)+b_h(\boldsymbol{\Psi}_h,I^{\rm CR}_h\boldsymbol{u}-\boldsymbol{u}_h)	= a(I_{h}^{\div}\boldsymbol{\Phi}-\boldsymbol{\Phi},\boldsymbol{\Psi}_h) \\
%\notag
%&\leq \iota^{-2}\|I_{h}^{\div}\boldsymbol{\Phi}-\boldsymbol{\Phi}\|_{\widetilde{\mathcal{A}}}\|\boldsymbol{\Psi}_h\|_{\widetilde{\mathcal{A}}}\lesssim \iota^{-2}|\boldsymbol{\Phi}|_1\|\boldsymbol{\Psi}_h\|_{\widetilde{\mathcal{A}}}\lesssim h|\boldsymbol{u}|_3\|\boldsymbol{\Psi}_h\|_{\widetilde{\mathcal{A}}}.
%\end{align}
\begin{equation}\label{errresult-pf1}
	a(I_{h}^{\div}\boldsymbol{\Phi}-\boldsymbol{\Phi}_h,\boldsymbol{\Psi}_h)+b_h(\boldsymbol{\Psi}_h,I^{\rm CR}_h\boldsymbol{u}-\boldsymbol{u}_h)\lesssim \iota h|\boldsymbol{\sigma}(\boldsymbol{u})|_{2}(\iota^{-1}\|\boldsymbol{\Psi}_h\|_{\widetilde{\mathcal{A}}}).
\end{equation}
On the other hand, using \eqref{IhdivProp1}, \eqref{IhCRProp1}--\eqref{IhCRProp2}, and \eqref{eq:erreqn2}, we obtain
\begin{equation}
\label{errresult-pf2}
\begin{aligned}
&\quad\; b_h(I_{h}^{\div}\boldsymbol{\Phi}-\boldsymbol{\Phi}_h,\boldsymbol{v}_h)-c_h(I^{\rm CR}_h\boldsymbol{u}-\boldsymbol{u}_h,\boldsymbol{v}_h) \\
&= b_h(\boldsymbol{\Phi}-\boldsymbol{\Phi}_h,\boldsymbol{v}_h)-c_h(\boldsymbol{u}-\boldsymbol{u}_h,\boldsymbol{v}_h) +\sum_{F\in\mathcal F_h}h^{-1}_F([\![\boldsymbol{u}-I^{\rm CR}_h\boldsymbol{u}]\!],[\![\boldsymbol{v}_h]\!])_F\\
&\lesssim b_h(\boldsymbol{\Phi},\boldsymbol{v}_h)-c_h(\boldsymbol{u},\boldsymbol{v}_h)+(\boldsymbol{f}, \boldsymbol{v}_h) + h|\boldsymbol{u}|_2{\interleave \boldsymbol{v}_h\interleave}_{1,h}.
\end{aligned}
\end{equation}
The combination of \eqref{errresult-pf20260630}--\eqref{errresult-pf2} and the discrete stability \eqref{discretestability} yields
\begin{equation}\label{eq:error0}
\begin{aligned}
&\quad\;\iota^{-1}\|I^{\div}_{h}\boldsymbol{\Phi}-\boldsymbol{\Phi}_h\|_{\widetilde{\mathcal{A}}}+{\interleave I^{\rm CR}_h\boldsymbol{u}-\boldsymbol{u}_h\interleave}_{\iota,h} \\
& \lesssim \iota h|\boldsymbol{\sigma}(\boldsymbol{u})|_{2} + h|\boldsymbol{u}|_2 + \sup_{\boldsymbol{v}_h\in V_{h}}
\frac{b_h(\boldsymbol{\Phi},\boldsymbol{v}_h)-c_h(\boldsymbol{u},\boldsymbol{v}_h)+(\boldsymbol{f}, \boldsymbol{v}_h)}{\interleave \boldsymbol{v}_h\interleave_{\iota,h}}.
\end{aligned}
\end{equation}
Hence, the estimate \eqref{errresult1} follows from the triangle inequality, \eqref{IhCRerror1}, \eqref{consistencyerr-1}, and \eqref{stress-controls-u}.
%By \eqref{IhCRProp1} and \eqref{CRinfsup},
%\begin{equation*}
%	\|\div_h(I^{\rm CR}_h\boldsymbol{u}-\boldsymbol{u}_h)\|_0\lesssim\sup_{\boldsymbol{v}_h\in V_h}\dfrac{(\div_h(I^{\rm CR}_h\boldsymbol{u}-\boldsymbol{u}_h),\div_h\boldsymbol{v}_h)}{|\boldsymbol{v}_h|_{1,h}} = \sup_{\boldsymbol{v}_h\in V_h}\dfrac{(\div_h(\boldsymbol{u}_h-\boldsymbol{u}),\div_h\boldsymbol{v}_h)}{|\boldsymbol{v}_h|_{1,h}}.
%\end{equation*}
%It follows from \eqref{SGE-MFEM}, \eqref{consistencyerr-1} and \eqref{errresult1} that
%\begin{align*}
%	\lambda(\div_h(\boldsymbol{u}_h-\boldsymbol{u}),\div_h\boldsymbol{v}_h) &= 2\mu(\boldsymbol{\varepsilon}_h(\boldsymbol{u}-\boldsymbol{u}_h),\boldsymbol{\varepsilon}_h(\boldsymbol{v}_h))+ (\boldsymbol{f},\boldsymbol{v}_h)+b_h(\boldsymbol{\Phi}_h,\boldsymbol{v}_h)-c_h(\boldsymbol{u},\boldsymbol{v}_h)\\
%	&\lesssim h(\|\boldsymbol{\Phi}\|_{2}+\|\boldsymbol{\sigma}(\boldsymbol{u})\|_{1})|\boldsymbol{v}_h|_{1,h}.
%\end{align*}
%Therefore, \eqref{errresultdiv1} holds from the last two inequalities, the triangle inequality and \eqref{IhCRProp1}.
%Hence, by using the triangle inequality and \eqref{IhdivProp2}, we derive \eqref{errresult1} from \eqref{consistencyerr-1} and \eqref{IhdivProp2}.
\end{proof}

\begin{remark}\rm	
Under the stated regularity assumptions and for fixed parameters, \eqref{errresult1} is first order in $h$. This order is optimal for the lowest-order tensor-valued Raviart--Thomas approximation and the linear Crouzeix--Raviart approximation in their respective natural norms. The term $\iota{\interleave I_h^{\rm CR}\boldsymbol{u}-\boldsymbol{u}_h\interleave}_{2,h}$ possesses the supercloseness property, which motivates the local post-processing below.
\end{remark}

Although \eqref{errresult1} is optimal in $h$ for fixed parameters, its regularity factors may deteriorate as $\iota\to0$ in the presence of boundary layers. We next derive a complementary error estimate that is uniform with respect to both $\iota$ and $\lambda$.
\begin{lemma}
Let $(\boldsymbol{\Phi},\boldsymbol{u})\in H^{-1}(\div\div,\Omega; \mathbb{S}\otimes\mathbb{R}^d)\times H_0^1(\Omega;\mathbb{R}^d)$ be the solution of problem~\eqref{SGE-mix}, and let $\boldsymbol{u}_0\in H^1_0(\Omega;\mathbb{R}^d)$ solve the weak problem \eqref{SGElinear-weak}. Assume that Assumption~\ref{ass:uniform-regularity} holds. Then, for any $\boldsymbol{v}_h\in V_{h}$,
\begin{equation}
\label{consistencyerr-2}
b_h(\boldsymbol{\Phi},\boldsymbol{v}_h)-c_h(\boldsymbol{u},\boldsymbol{v}_h)+(\boldsymbol{f}, \boldsymbol{v}_h)\lesssim (\iota^{1/2}+h)\|\boldsymbol{f}\|_0|\boldsymbol{v}_h|_{1,h}.
\end{equation}
\end{lemma}
\begin{proof}
It follows from the Cauchy--Schwarz inequality, the norm equivalence \eqref{CRnormeq}, and the regularity estimates~\eqref{Regularity-u}--\eqref{Regularity-divu} that
\begin{equation*}
b_h(\boldsymbol{\Phi},\boldsymbol{v}_h)-c_h(\boldsymbol{u}-\boldsymbol{u}_0,\boldsymbol{v}_h)\lesssim(\iota^2|\boldsymbol{\sigma}(\boldsymbol{u})|_2+\|\boldsymbol{\sigma}(\boldsymbol{u}-\boldsymbol{u}_0)\|_0)|\boldsymbol{v}_h|_{1,h}\lesssim\iota^{1/2}\|\boldsymbol{f}\|_0|\boldsymbol{v}_h|_{1,h}.
\end{equation*}
Since
\begin{equation*}
b_h(\boldsymbol{\Phi},\boldsymbol{v}_h)-c_h(\boldsymbol{u},\boldsymbol{v}_h)+(\boldsymbol{f}, \boldsymbol{v}_h) = b_h(\boldsymbol{\Phi},\boldsymbol{v}_h)-c_h(\boldsymbol{u}-\boldsymbol{u}_0,\boldsymbol{v}_h)-c_h(\boldsymbol{u}_0,\boldsymbol{v}_h)+(\boldsymbol{f}, \boldsymbol{v}_h), 
\end{equation*} 
it suffices to prove
\begin{equation}\label{consistencyerr-2-pf1}
(\boldsymbol{f}, \boldsymbol{v}_h)-c_h(\boldsymbol{u}_0,\boldsymbol{v}_h)\lesssim h\|\boldsymbol{f}\|_{0}|\boldsymbol{v}_h|_{1,h}.
\end{equation}
By \eqref{elasregularity}, $\boldsymbol{u}_0\in H^2(\Omega;\mathbb{R}^d)$. Hence the weak problem \eqref{SGElinear-weak} yields $-\div\boldsymbol{\sigma}(\boldsymbol{u}_0)=\boldsymbol{f}$ in $L^2(\Omega;\mathbb{R}^d)$. Elementwise integration by parts and the weak continuity \eqref{weak-c} give
\begin{align*}
(\boldsymbol{f}, \boldsymbol{v}_h) - c_h(\boldsymbol{u}_0,\boldsymbol{v}_h) =&-\sum_{T\in \mathcal{T}_h}\sum_{F\in\mathcal{F}(T)}(\boldsymbol{\sigma}(\boldsymbol{u}_0)\boldsymbol{n},\boldsymbol{v}_h)_{F}\\
=&-\sum_{F\in\mathcal{F}_h}(\boldsymbol{\sigma}(\boldsymbol{u}_0)\boldsymbol{n}-Q_{0,F}(\boldsymbol{\sigma}(\boldsymbol{u}_0)\boldsymbol{n}),[\![\boldsymbol{v}_h]\!])_{F}.
\end{align*}
It follows from the error estimate of $Q_{0,F}$ and the trace inequality that
\begin{align*}
(\boldsymbol{f}, \boldsymbol{v}_h) - c_h(\boldsymbol{u}_0,\boldsymbol{v}_h)\lesssim h|\boldsymbol{\sigma}(\boldsymbol{u}_0)|_1\bigg(\sum_{F\in \mathcal{F}_h}h_F^{-1}\|[\![\boldsymbol{v}_h]\!]\|_{0,F}^2\bigg)^{1/2},
\end{align*}
which together with \eqref{elasregularity} and \eqref{CRnormeq} implies \eqref{consistencyerr-2-pf1}.
\end{proof}

\begin{theorem}\label{thm:errresult23}
Let $(\boldsymbol{\Phi},\boldsymbol{u})\in H^{-1}(\div\div,\Omega; \mathbb{S}\otimes\mathbb{R}^d)\times H_0^1(\Omega;\mathbb{R}^d)$ and $(\boldsymbol{\Phi}_h,\boldsymbol{u}_h)\in \Sigma_{h}\times V_{h}$ be the solutions of problem \eqref{SGE-mix} and the mixed FEM \eqref{SGE-MFEM}, respectively. Assume that Assumption~\ref{ass:uniform-regularity} holds. Then
\begin{align}
\label{errresult2}
\iota^{-1}\|\boldsymbol{\Phi}-\boldsymbol{\Phi}_h\|_{\widetilde{\mathcal{A}}}+{\interleave \boldsymbol{u}-\boldsymbol{u}_h\interleave}_{\iota,h}&\lesssim (\iota^{1/2}+h)\|\boldsymbol{f}\|_{0},\\
\label{errresult3}
{\interleave \boldsymbol{u}_0-\boldsymbol{u}_h\interleave}_{1,h}&\lesssim (\iota^{1/2}+h)\|\boldsymbol{f}\|_{0}.
\end{align}
\end{theorem}
\begin{proof}
Using the estimate \eqref{IhdivProp2} and the regularity estimates \eqref{Regularity-u}--\eqref{Regularity-divu}, we obtain
\begin{equation}
\label{errresult-pf4}
\|I^{\div}_{h}\boldsymbol{\Phi}-\boldsymbol{\Phi}\|_0\lesssim \|\boldsymbol{\Phi}\|_0\lesssim\iota^2\|\boldsymbol{\sigma}(\boldsymbol{u})\|_1\lesssim \iota^{3/2}\|\boldsymbol{f}\|_0.
\end{equation}
This together with \eqref{eq:erreqn1} yields
\begin{equation*}
a(I_{h}^{\div}\boldsymbol{\Phi}-\boldsymbol{\Phi}_h,\boldsymbol{\Psi}_h)+b_h(\boldsymbol{\Psi}_h,I^{\rm CR}_h\boldsymbol{u}-\boldsymbol{u}_h)\lesssim \iota^{1/2}\|\boldsymbol{f}\|_0(\iota^{-1}\|\boldsymbol{\Psi}_h\|_{\widetilde{\mathcal{A}}}).
\end{equation*}
On the other hand, by \eqref{errresult-pf2}, the weak continuity \eqref{weak-c}, \eqref{consistencyerr-2}, and \eqref{IhCRerror2},
\begin{equation*}
% \label{errresult-pf27}
\begin{aligned}
&\quad b_h(I_{h}^{\div}\boldsymbol{\Phi}-\boldsymbol{\Phi}_h,\boldsymbol{v}_h)-c_h(I^{\rm CR}_h\boldsymbol{u}-\boldsymbol{u}_h,\boldsymbol{v}_h) \\
&\lesssim b_h(\boldsymbol{\Phi},\boldsymbol{v}_h)-c_h(\boldsymbol{u},\boldsymbol{v}_h)+(\boldsymbol{f}, \boldsymbol{v}_h) + {\interleave\boldsymbol{u}-I^{\rm CR}_h\boldsymbol{u}\interleave}_{1,h}{\interleave \boldsymbol{v}_h\interleave}_{1,h} \\
&\lesssim (\iota^{1/2}+h)\|\boldsymbol{f}\|_0{\interleave \boldsymbol{v}_h\interleave}_{1,h}.
\end{aligned}
\end{equation*}
By substituting the last two inequalities into the discrete stability estimate \eqref{discretestability}, we obtain
\begin{equation}\label{errresult-pf5}
\iota^{-1}\|I^{\div}_{h}\boldsymbol{\Phi}-\boldsymbol{\Phi}_h\|_{\widetilde{\mathcal{A}}}+{\interleave I^{\rm CR}_h\boldsymbol{u}-\boldsymbol{u}_h\interleave}_{\iota,h}\lesssim (\iota^{1/2}+h)\|\boldsymbol{f}\|_0.
\end{equation}
The estimate \eqref{errresult2} follows from the triangle inequality, \eqref{errresult-pf4}--\eqref{errresult-pf5} and \eqref{IhCRerror2}.

Finally, the estimate \eqref{errresult3} follows from the triangle inequality, \eqref{errresult2}, and the regularity estimates \eqref{Regularity-u}--\eqref{Regularity-divu}.
\end{proof}

\begin{remark}\rm
The constants in \eqref{errresult2}--\eqref{errresult3} are independent of $h$, $\iota$, and $\lambda$. Thus, these estimates remain uniform in both the singular perturbation limit $\iota\to0$ and the nearly incompressible limit $\lambda\to\infty$. In the boundary-layer regime $\iota^{1/2}\lesssim h$, \eqref{errresult2} yields an $O(h)$ parameter-uniform bound. Moreover, \eqref{errresult3} shows that the error between the discrete solution and the reduced elasticity solution is controlled simultaneously by the perturbation scale $\iota^{1/2}$ and the discretization scale $h$. Hence, \eqref{errresult1} and \eqref{errresult2} furnish complementary optimal-order and parameter-robust error bounds.
\end{remark}

\subsection{Local post-processing}
We next construct a post-processed approximation $\boldsymbol{u}_h^*$ and derive its error estimate in a parameter-dependent broken $H^2$ seminorm.

We introduce the second-order Brezzi--Douglas--Marini element \cite{MR3097958, MR4458899}. The second-order BDM element takes $\mathbb{P}_{2}(T;\mathbb{R}^{d})$ as the shape function space, and the DoFs are chosen as
\begin{align}
\label{BDM-dof1}
(\boldsymbol{v}\cdot \boldsymbol n, q)_F\quad &\forall \ q\in \mathbb{P}_{2}(F), \ F\in\mathcal{F}(T),\\
\label{BDM-dof2}
(\boldsymbol{v}, \boldsymbol{q})_{T} \quad &\forall \ \boldsymbol{q}\in 
%\nabla\mathbb{P}_1(T)\oplus\mathbb{P}_0(T;\mathbb{K})\boldsymbol{x}.
\mathbb{P}_1(T;\mathbb{R}^d)~\textrm{satisfying}~ \boldsymbol{x}\cdot\boldsymbol{q}\in\mathbb{P}_1(T).
\end{align}
Let $I_T^{\rm BDM}:H^1(T;\mathbb{R}^{d})\rightarrow\mathbb{P}_{2}(T;\mathbb{R}^{d})$ be the canonical interpolation operator based on DoFs (\ref{BDM-dof1})--(\ref{BDM-dof2}). From \cite{MR3097958}, it holds
\begin{align}
\label{BDM-div}
\div(I_T^{\rm BDM}\boldsymbol{v})=Q_{1,T}(\div\boldsymbol{v}) \quad\forall \ \boldsymbol{v}\in H^1(T;\mathbb{R}^{d}).
\end{align}
Define a new approximation $\boldsymbol{u}^*_h\in\mathbb{P}_2(\mathcal{T}_h;\mathbb{R}^d)$ to $\boldsymbol{u}$ as a solution of the following problem: for each $T\in\mathcal{T}_h$,
\begin{subequations}\label{SGE-Post}
\begin{align}
\label{SGE-Post1}
(\boldsymbol{u}^*_h,\boldsymbol{q})_T &= (\boldsymbol{u}_h,\boldsymbol{q})_T \quad\quad\quad\quad\;\;\, \forall \ \boldsymbol{q}\in  \mathbb{P}_1(T;\mathbb{R}^d), \\
\label{SGE-Post2}
\iota^2(\grad\boldsymbol{\sigma}(\boldsymbol{u}^*_h),\grad\boldsymbol{\varepsilon}(\boldsymbol{q}))_T &= (\boldsymbol{\Phi}_h,\grad\boldsymbol{\varepsilon}(\boldsymbol{q}))_T \quad \forall \ \boldsymbol{q}\in \mathbb{P}_2(T;\mathbb{R}^d).
\end{align}
\end{subequations}
This local problem is well-posed. 
% Indeed, for
% \[
% d_T(\boldsymbol{p},\boldsymbol{q}):=(\grad\boldsymbol{\sigma}(\boldsymbol{p}),\grad\boldsymbol{\varepsilon}(\boldsymbol{q}))_T,
% \]
% we have
% \[
% d_T(\boldsymbol{p},\boldsymbol{p})=2\mu|\boldsymbol{\varepsilon}(\boldsymbol{p})|_{1,T}^2+\lambda|\div\boldsymbol{p}|_{1,T}^2,
% \qquad \boldsymbol{p}\in\mathbb P_2(T;\mathbb R^d).
% \]
% Hence, $\ker d_T=\mathbb P_1(T;\mathbb R^d)$: the nontrivial implication follows from $\grad\boldsymbol{\varepsilon}(\boldsymbol{p})=0$ and the identity for the second derivatives of $\boldsymbol{p}$. Since $\grad\boldsymbol{\varepsilon}(\boldsymbol{q})=0$ for every $\boldsymbol{q}\in\mathbb P_1(T;\mathbb R^d)$, the right-hand side of \eqref{SGE-Post2} vanishes on $\mathbb P_1(T;\mathbb R^d)$ and thus defines a functional on the quotient. The bilinear form $d_T$ is coercive on $\mathbb P_2(T;\mathbb R^d)/\mathbb P_1(T;\mathbb R^d)$, whereas \eqref{SGE-Post1} uniquely determines the $\mathbb P_1(T;\mathbb R^d)$ component.
\begin{theorem}[Optimal post-processing]
Let $(\boldsymbol{\Phi},\boldsymbol{u})\in H^{-1}(\div\div,\Omega; \mathbb{S}\otimes\mathbb{R}^d)\times H_0^1(\Omega;\mathbb{R}^d)$ and $(\boldsymbol{\Phi}_h,\boldsymbol{u}_h)\in \Sigma_{h}\times V_{h}$ be the solution of problem \eqref{SGE-mix} and the mixed FEM \eqref{SGE-MFEM}, respectively. If $\boldsymbol{\Phi}\in H^2(\Omega;\mathbb{S}\otimes\mathbb{R}^d)$ and $\boldsymbol{u}\in H^3(\Omega;\mathbb{R}^d)$, then
\begin{align}
\label{errresult-post}
\iota(|\boldsymbol{\varepsilon}_h(\boldsymbol{u}-\boldsymbol{u}^*_h)|_{1,h}+\sqrt{\lambda}|\div_h(\boldsymbol{u}-\boldsymbol{u}^*_h)|_{1,h})\lesssim h(|\div\boldsymbol{\Phi}|_{1}+\iota|\boldsymbol{\sigma}(\boldsymbol{u})|_{2}+|\boldsymbol{\sigma}(\boldsymbol{u})|_{1}).
\end{align}
\end{theorem}
\begin{proof}
Set $\boldsymbol{w}=I_T^{\rm BDM}\boldsymbol{u}-\boldsymbol{u}^*_h\in\mathbb{P}_2(T;\mathbb{R}^d)$ for simplicity. It follows from \eqref{SGE1} and \eqref{SGE-Post2} that
\begin{equation*}
\iota^2(\grad\boldsymbol{\sigma}(\boldsymbol{u}-\boldsymbol{u}^*_h),\grad\boldsymbol{\varepsilon}(\boldsymbol{w}))_T=(\boldsymbol{\Phi}-\boldsymbol{\Phi}_h,\grad\boldsymbol{\varepsilon}(\boldsymbol{w}))_T.
\end{equation*}
By the definition of $\boldsymbol{w}$,
\begin{align*}
\iota^2(2\mu|\boldsymbol{\varepsilon}(\boldsymbol{w})|^2_{1,T}+\lambda|\div\boldsymbol{w}|^2_{1,T})&=\iota^2(\grad\boldsymbol{\sigma}(I_T^{\rm BDM}\boldsymbol{u}-\boldsymbol{u}),\grad\boldsymbol{\varepsilon}(\boldsymbol{w}))_T\\
&\quad+(\boldsymbol{\Phi}-\boldsymbol{\Phi}_h,\grad\boldsymbol{\varepsilon}(\boldsymbol{w}))_T,
\end{align*}
By the Cauchy--Schwarz inequality in the $\widetilde{\mathcal{A}}$-inner product,
\begin{align*}
\big|\big(\boldsymbol{\Phi}-\boldsymbol{\Phi}_h,\grad\boldsymbol{\varepsilon}(\boldsymbol{w})\big)_T\big|
&\leq \|\boldsymbol{\Phi}-\boldsymbol{\Phi}_h\|_{\widetilde{\mathcal{A}},T}\,
\|\grad\boldsymbol{\varepsilon}(\boldsymbol{w})\|_{\widetilde{\mathcal{A}}^{-1},T},
\end{align*}
where $\widetilde{\mathcal{A}}^{-1}\grad\boldsymbol{\varepsilon}(\boldsymbol{w})=\grad\boldsymbol{\sigma}(\boldsymbol{w})$. Applying the same inequality to the first term on the right-hand side and using \eqref{BDM-div}, we obtain
\begin{align*}
\iota(|\boldsymbol{\varepsilon}(\boldsymbol{w})|_{1,T}+\sqrt{\lambda}|\div\boldsymbol{w}|_{1,T})&\lesssim \iota(|I_T^{\rm BDM}\boldsymbol{u}-\boldsymbol{u}|_{2,T}+\sqrt{\lambda}|Q_{1,T}\div\boldsymbol{u}-\div\boldsymbol{u}|_{1,T})\\
&\quad+\iota^{-1}\big(\widetilde{\mathcal A}(\boldsymbol{\Phi}-\boldsymbol{\Phi}_h),\boldsymbol{\Phi}-\boldsymbol{\Phi}_h\big)_T^{1/2}.
\end{align*}
The standard approximation estimates
\[
|I_T^{\rm BDM}\boldsymbol{u}-\boldsymbol{u}|_{2,T}\lesssim h_T|\boldsymbol{u}|_{3,T},
\qquad
|Q_{1,T}\div\boldsymbol{u}-\div\boldsymbol{u}|_{1,T}\lesssim h_T|\div\boldsymbol{u}|_{2,T}
\]
together with the triangle inequality, \eqref{errresult1}, \eqref{BDM-div}, and \eqref{stress-controls-u} yield \eqref{errresult-post}.
\end{proof}

\section{Hybridization}\label{sec4}

This section introduces a hybridized formulation of \eqref{SGE-MFEM} by relaxing the normal continuity of the tensor variable and introducing a face unknown. 
% We establish its equivalence with the original method, derive regularity-dependent and parameter-robust error estimates, and obtain a symmetric positive-definite system after local elimination of the tensor variable.

\subsection{Weak gradient and norm equivalence}
To this end, we introduce the broken finite element space
\begin{align*}
\Sigma^{-1}_{h}:=\{\boldsymbol{\Psi}\in L^{2}(\Omega;\mathbb{S}\otimes\mathbb{R}^d):\boldsymbol{\Psi}|_T\in \Sigma(T;\mathbb{S}\otimes\mathbb{R}^d)~\textrm{for}~\textrm{each}~T\in\mathcal{T}_h\},
\end{align*}
and the multiplier space $M_h:=\mathbb{P}_0(\mathcal{T}_h;\mathbb{S})\times\mathbb{P}_0(\mathring{\mathcal{F}_h};\mathbb{S})$. Here and below, functions in $\mathbb{P}_0(\mathring{\mathcal{F}_h};\mathbb{S})$ are extended by zero to boundary faces.

Define the weak gradient operator $\grad_w:M_h\rightarrow\Sigma^{-1}_{h}$ as follows: for $\boldsymbol{\tau}_h=(\boldsymbol{\tau}_0,\boldsymbol{\tau}_b)\in M_h$, $\grad_w\boldsymbol{\tau}_h\in \Sigma^{-1}_{h}$ is elementwise determined by
\begin{align*}
(\grad_w\boldsymbol{\tau}_h,\boldsymbol{\Psi})_T &= -(\div\boldsymbol{\Psi},\boldsymbol{\tau}_0)_T+(\boldsymbol{\Psi}\boldsymbol{n}_{\partial T},\boldsymbol{\tau}_b)_{\partial T} \\
&=(\boldsymbol{\Psi}\boldsymbol{n}_{\partial T},\boldsymbol{\tau}_b-\boldsymbol{\tau}_0)_{\partial T}, \qquad\qquad\quad \forall \ \boldsymbol{\Psi}\in \Sigma(T;\mathbb{S}\otimes\mathbb{R}^d), T\in\mathcal{T}_h.
\end{align*}
For $\boldsymbol{\tau}\in H_0^1(\Omega;\mathbb{S})$, define $Q_M\boldsymbol{\tau}\in M_h$ by local $L^2$-projection
\begin{equation*}
Q_M\boldsymbol{\tau}=(Q_{0,h}\boldsymbol{\tau},Q_{0,\mathcal{F}_h}\boldsymbol{\tau}).
\end{equation*}
Since $\div\boldsymbol{\Psi}\in\mathbb{P}_0(T;\mathbb{S})$ and $\boldsymbol{\Psi}\boldsymbol{n}_{\partial T}|_F\in\mathbb{P}_0(F;\mathbb{S})$ for every $\boldsymbol{\Psi}\in\Sigma(T;\mathbb{S}\otimes\mathbb{R}^d)$, the projection orthogonality and elementwise integration by parts give
\begin{equation}
\label{hybrid-change}
\grad_w(Q_M\boldsymbol{\tau})=Q_{\Sigma}(\grad\boldsymbol{\tau}),
\end{equation}
where $Q_{\Sigma}$ is the $L^2$-projection to the space $\Sigma^{-1}_h$.
Since $\Sigma^{-1}_h$ contains the piecewise constant third-order tensors, the elementwise projection satisfies
\begin{equation}
\label{Qsigma-error}
\|\boldsymbol{\Psi}-Q_\Sigma\boldsymbol{\Psi}\|_0\lesssim h^s|\boldsymbol{\Psi}|_s \quad \forall~\boldsymbol{\Psi}\in H^s(\Omega;\mathbb S\otimes\mathbb R^d),\quad 0\leq s\leq 1.
\end{equation}

We equip $M_h$ with the norm
\begin{equation*}
\|\boldsymbol{\tau}_h\|_{1,h}^2:=\sum_{T\in\mathcal{T}_h}h^{-1}_T(2\mu\|\boldsymbol{\tau}_0-\boldsymbol{\tau}_b\|^2_{\partial T}+\lambda\|\tr(\boldsymbol{\tau}_0-\boldsymbol{\tau}_b)\|^2_{\partial T}), \quad \forall \ \boldsymbol{\tau}_h\in M_h.
\end{equation*}
Recall that $\mathcal{C} \boldsymbol{\tau} := 2\mu\boldsymbol{\tau} + \lambda\tr(\boldsymbol{\tau}) \boldsymbol{I}$ for $\boldsymbol{\tau}\in\mathbb{S}$. For any $\boldsymbol{\tau}_h\in M_h$, we have
\begin{equation*}
\|\boldsymbol{\tau}_h\|_{1,h}^2 = \sum_{T\in\mathcal{T}_h}h^{-1}_T\|\boldsymbol{\tau}_0-\boldsymbol{\tau}_b\|_{\mathcal{C},\partial T}^2,
\end{equation*}
where $\|\boldsymbol{\varsigma}\|_{\mathcal{C},\partial T}^2:=(\mathcal{C}\boldsymbol{\varsigma},\boldsymbol{\varsigma})_{\partial T}$.

\begin{lemma}
The following norm equivalence holds:
\begin{equation}\label{normequiv}
\|\grad_w\boldsymbol{\tau}_h\|_{\widetilde{\mathcal{C}}}\eqsim\|\boldsymbol{\tau}_h\|_{1,h} \quad \forall \ \boldsymbol{\tau}_h\in M_h.
\end{equation}
Here, the squared weighted norm is defined by $\|\boldsymbol{\Psi}\|_{\widetilde{\mathcal{C}}}^2:=(\widetilde{\mathcal{C}}\boldsymbol{\Psi},\boldsymbol{\Psi})$.
\end{lemma}
\begin{proof}
Let $T\in\mathcal{T}_h$.
Since $\widetilde{\mathcal{C}}$ maps $\Sigma(T;\mathbb{S}\otimes\mathbb{R}^d)$ into itself by \eqref{eq:Ctensorprop}, $\widetilde{\mathcal{C}}\grad_w\boldsymbol{\tau}_h$ is an admissible test function in the definition of $\grad_w\boldsymbol{\tau}_h$. Hence, by the definition of $\grad_w\boldsymbol{\tau}_h$ and \eqref{eq:Ctildetensorprop},
% \begin{align*}
% \|\grad_w\boldsymbol{\tau}_h\|_{\widetilde{\mathcal{C}}}^2&=(\widetilde{\mathcal{C}}(\grad_w\boldsymbol{\tau}_h), \grad_w\boldsymbol{\tau}_h)= \sum_{T\in\mathcal{T}_h}\big(\big(\widetilde{\mathcal{C}}(\grad_w\boldsymbol{\tau}_h)\big)\boldsymbol{n},\boldsymbol{\tau}_b-\boldsymbol{\tau}_0\big)_{\partial T} \\
% &= \sum_{T\in\mathcal{T}_h}\big(\mathcal{C}\big((\grad_w\boldsymbol{\tau}_h)\boldsymbol{n}\big),\boldsymbol{\tau}_b-\boldsymbol{\tau}_0\big)_{\partial T}\leq \sum_{T\in\mathcal{T}_h}\|(\grad_w\boldsymbol{\tau}_h)\boldsymbol{n}\|_{\mathcal{C},\partial T}\|\boldsymbol{\tau}_0-\boldsymbol{\tau}_b\|_{\mathcal{C},\partial T}
% \end{align*}
\begin{align*}
(\widetilde{\mathcal{C}}(\grad_w\boldsymbol{\tau}_h), \grad_w\boldsymbol{\tau}_h)_T&= \big(\big(\widetilde{\mathcal{C}}(\grad_w\boldsymbol{\tau}_h)\big)\boldsymbol{n}_{\partial T},\boldsymbol{\tau}_b-\boldsymbol{\tau}_0\big)_{\partial T} \\
&= \big(\mathcal{C}\big((\grad_w\boldsymbol{\tau}_h)\boldsymbol{n}_{\partial T}\big),\boldsymbol{\tau}_b-\boldsymbol{\tau}_0\big)_{\partial T} \\
&\leq \|(\grad_w\boldsymbol{\tau}_h)\boldsymbol{n}_{\partial T}\|_{\mathcal{C},\partial T}\|\boldsymbol{\tau}_0-\boldsymbol{\tau}_b\|_{\mathcal{C},\partial T}.
\end{align*}
Applying \eqref{eq:Ctildetensorprop} again, we get from the inverse inequality that
\begin{align*}
\|(\grad_w\boldsymbol{\tau}_h)\boldsymbol{n}_{\partial T}\|_{\mathcal{C},\partial T}^2&=\big(\mathcal{C}((\grad_w\boldsymbol{\tau}_h)\boldsymbol{n}_{\partial T}), (\grad_w\boldsymbol{\tau}_h)\boldsymbol{n}_{\partial T}\big)_{\partial T} \\
&= \big(\big(\widetilde{\mathcal{C}}(\grad_w\boldsymbol{\tau}_h)\big)\boldsymbol{n}_{\partial T},(\grad_w\boldsymbol{\tau}_h)\boldsymbol{n}_{\partial T}\big)_{\partial T} \\
&\leq \big(\widetilde{\mathcal{C}}(\grad_w\boldsymbol{\tau}_h),\grad_w\boldsymbol{\tau}_h\big)_{\partial T} \\
&\lesssim h_T^{-1}\big(\widetilde{\mathcal{C}}(\grad_w\boldsymbol{\tau}_h),\grad_w\boldsymbol{\tau}_h\big)_{T}.
\end{align*}
A combination of the last two inequalities gives
\begin{equation*}
\|\grad_w\boldsymbol{\tau}_h\|_{\widetilde{\mathcal{C}}}\lesssim\|\boldsymbol{\tau}_h\|_{1,h}.
\end{equation*}

For the reverse inequality, let $\boldsymbol{\Psi}_h\in \Sigma^{-1}_{h}$ be locally determined by the normal traces
\begin{equation*}
(\boldsymbol{\Psi}_h\boldsymbol{n}_{\partial T})|_{\partial T}=h_T^{-1}\mathcal{C}(\boldsymbol{\tau}_b-\boldsymbol{\tau}_0)|_{\partial T} \quad\quad \forall \ T\in\mathcal{T}_h.
\end{equation*}
Similar to the proof of Lemma \ref{inf-supH2-lemma}, we have
\begin{align*}
\|\boldsymbol{\Psi}_h\|_{\widetilde{\mathcal{A}}}\lesssim\|\boldsymbol{\tau}_h\|_{1,h}, \qquad
(\grad_w\boldsymbol{\tau}_h,\boldsymbol{\Psi}_h)=\sum_{T\in\mathcal{T}_h}(\boldsymbol{\Psi}_h\boldsymbol{n}_{\partial T},\boldsymbol{\tau}_b-\boldsymbol{\tau}_0)_{\partial T}=\|\boldsymbol{\tau}_h\|^2_{1,h}.
\end{align*}
The Cauchy--Schwarz inequality in the $\widetilde{\mathcal{C}}$--$\widetilde{\mathcal{A}}$ pairing now yields
\begin{align*}
\|\boldsymbol{\tau}_h\|^2_{1,h}
&=(\grad_w\boldsymbol{\tau}_h,\boldsymbol{\Psi}_h)
\leq \|\grad_w\boldsymbol{\tau}_h\|_{\widetilde{\mathcal{C}}}\|\boldsymbol{\Psi}_h\|_{\widetilde{\mathcal{A}}}
\lesssim \|\grad_w\boldsymbol{\tau}_h\|_{\widetilde{\mathcal{C}}}\|\boldsymbol{\tau}_h\|_{1,h}.
\end{align*}
This proves the reverse inequality and completes the proof.
\end{proof}

\subsection{Hybridized formulation}
The hybridized formulation of the mixed finite element method \eqref{SGE-MFEM} seeks $(\boldsymbol{u}_h,\boldsymbol{\eta}_h)\in V_{h}\times\mathbb{P}_0(\mathring{\mathcal{F}_h};\mathbb{S})$ such that
\begin{equation}\label{SGE-wg}
	\iota^2\big(\widetilde{\mathcal{C}}\grad_w(\boldsymbol{\varepsilon}_h(\boldsymbol{u}_h),\boldsymbol{\eta}_h),\grad_w(\boldsymbol{\varepsilon}_h(\boldsymbol{v}_h),\boldsymbol{\xi}_h)\big)+c_h(\boldsymbol{u}_h,\boldsymbol{v}_h) = (\boldsymbol{f},\boldsymbol{v}_h), 
\end{equation}
for any $\boldsymbol{v}_h\in V_{h}$ and $\boldsymbol{\xi}_h\in \mathbb{P}_0(\mathring{\mathcal{F}_h};\mathbb{S})$.
The bilinear form in \eqref{SGE-wg} is symmetric. Moreover, \eqref{eq:chcoercivity} and \eqref{normequiv} imply that the bilinear form is positive definite, and the hybridized method \eqref{SGE-wg} is well-posed.

The following theorem demonstrates the equivalence between the hybridized method \eqref{SGE-wg} and the mixed method \eqref{SGE-MFEM}.
\begin{theorem}
Let $(\boldsymbol{u}_h,\boldsymbol{\eta}_h)\in V_{h}\times\mathbb{P}_0(\mathring{\mathcal{F}_h};\mathbb{S})$ solve the hybridized method \eqref{SGE-wg}. Set $\boldsymbol{\Phi}_h:=\iota^2\widetilde{\mathcal{C}}\grad_w(\boldsymbol{\varepsilon}_h(\boldsymbol{u}_h),\boldsymbol{\eta}_h)$. Then $\boldsymbol{\Phi}_h\in\Sigma_{h}$, and $(\boldsymbol{\Phi}_h, \boldsymbol{u}_h)$ solves the mixed method \eqref{SGE-MFEM}.
\end{theorem}
\begin{proof}
Since $\widetilde{\mathcal{C}}$ leaves the local space $\Sigma(T;\mathbb{S}\otimes\mathbb{R}^d)$ invariant, we first have $\boldsymbol{\Phi}_h\in\Sigma_h^{-1}$.
Choosing $\boldsymbol{v}_h=0$, equation \eqref{SGE-wg} reduces to 
\begin{equation*}
\big(\boldsymbol{\Phi}_h,\grad_w(0,\boldsymbol{\xi}_h)\big)=0,\quad\forall\,\boldsymbol{\xi}_h\in \mathbb{P}_0(\mathring{\mathcal{F}_h};\mathbb{S}).
\end{equation*}
By the definition of $\grad_w$,
\begin{equation*}
\sum_{T\in\mathcal{T}_h}(\boldsymbol{\Phi}_h\boldsymbol{n}_{\partial T},\boldsymbol{\xi}_h)_{\partial T}=0,\quad\forall\,\boldsymbol{\xi}_h\in \mathbb{P}_0(\mathring{\mathcal{F}_h};\mathbb{S}).
\end{equation*}
Thus, $\boldsymbol{\Phi}_h\in\Sigma_{h}$.
 
Taking $\boldsymbol{\xi}_h=0$ in \eqref{SGE-wg} and using the definition of $\grad_w$ gives \eqref{SGE-MFEM2}. For $\boldsymbol{\Psi}_h\in  \Sigma_{h}$, the multiplier terms cancel facewise because $\boldsymbol{\Psi}_h\boldsymbol{n}_{F}$ is continuous and $\boldsymbol{\eta}_h$ vanishes on boundary faces. Hence,
\begin{equation*}
a(\boldsymbol{\Phi}_h,\boldsymbol{\Psi}_h)=(\grad_w(\boldsymbol{\varepsilon}_h(\boldsymbol{u}_h),\boldsymbol{\eta}_h),\boldsymbol{\Psi}_h)=-b_h(\boldsymbol{\Psi}_h,\boldsymbol{u}_h),
\end{equation*}
which is exactly \eqref{SGE-MFEM1}. Thus, $(\boldsymbol{\Phi}_h, \boldsymbol{u}_h)$ solves the mixed method \eqref{SGE-MFEM}. Since both formulations are uniquely solvable, their tensor and displacement components coincide.
\end{proof}

\begin{theorem}
Let $(\boldsymbol{u}_h,\boldsymbol{\eta}_h)$ solve \eqref{SGE-wg}, and let $\boldsymbol{\Phi}_h$ be defined as in the preceding theorem. If $\boldsymbol{\Phi}\in H^2(\Omega;\mathbb{S}\otimes\mathbb{R}^d)$ and $\boldsymbol{u}\in H^3(\Omega;\mathbb{R}^d)$, then
\begin{equation}
\label{errresult2-hybrid}
\iota\| Q_M\boldsymbol{\varepsilon}(\boldsymbol{u})-(\boldsymbol{\varepsilon}_h(\boldsymbol{u}_h),\boldsymbol{\eta}_h)\|_{1,h} 
\lesssim h(|\div\boldsymbol{\Phi}|_{1}+\iota|\boldsymbol{\sigma}(\boldsymbol{u})|_{2}+|\boldsymbol{\sigma}(\boldsymbol{u})|_{1}).
\end{equation}
If Assumption~\ref{ass:uniform-regularity} holds, then
\begin{equation}
\label{errresult1-hybrid}
\iota\| Q_M\boldsymbol{\varepsilon}(\boldsymbol{u})-(\boldsymbol{\varepsilon}_h(\boldsymbol{u}_h),\boldsymbol{\eta}_h)\|_{1,h}\lesssim(\iota^{1/2}+h)\|\boldsymbol{f}\|_{0}.
\end{equation}
\end{theorem}

\begin{proof}
Notice that $\widetilde{\mathcal{C}}$ is self-adjoint and leaves $\Sigma_h^{-1}$ invariant. Hence, for any $\boldsymbol{X}\in L^2(\Omega;\mathbb{S}\otimes\mathbb{R}^d)$ and $\boldsymbol{\Psi}_h\in\Sigma_h^{-1}$,
\begin{align*}
(Q_\Sigma\widetilde{\mathcal{C}}\boldsymbol{X},\boldsymbol{\Psi}_h)
&=(\widetilde{\mathcal{C}}\boldsymbol{X},\boldsymbol{\Psi}_h)
=(\boldsymbol{X},\widetilde{\mathcal{C}}\boldsymbol{\Psi}_h)
=(Q_\Sigma\boldsymbol{X},\widetilde{\mathcal{C}}\boldsymbol{\Psi}_h)
=(\widetilde{\mathcal{C}}Q_\Sigma\boldsymbol{X},\boldsymbol{\Psi}_h).
\end{align*}
Therefore, $Q_\Sigma\widetilde{\mathcal{C}}=\widetilde{\mathcal{C}}Q_\Sigma$. Since $\boldsymbol{\Phi}=\iota^2\widetilde{\mathcal{C}}\grad\boldsymbol{\varepsilon}(\boldsymbol{u})$, it follows from \eqref{hybrid-change} that
\begin{equation*}
Q_{\Sigma}\boldsymbol{\Phi}=\iota^2\widetilde{\mathcal{C}}Q_{\Sigma}(\grad\boldsymbol{\varepsilon}(\boldsymbol{u}))=\iota^2\widetilde{\mathcal{C}}\grad_w(Q_M\boldsymbol{\varepsilon}(\boldsymbol{u})).
\end{equation*}
Then
\begin{equation*}
Q_{\Sigma}\boldsymbol{\Phi}-\boldsymbol{\Phi}_h=\iota^2\widetilde{\mathcal{C}}\grad_w\big(Q_M\boldsymbol{\varepsilon}(\boldsymbol{u}) - (\boldsymbol{\varepsilon}_h(\boldsymbol{u}_h),\boldsymbol{\eta}_h)\big).
\end{equation*}
This implies
\begin{equation*}
\iota\|\grad_w\big(Q_M\boldsymbol{\varepsilon}(\boldsymbol{u}) - (\boldsymbol{\varepsilon}_h(\boldsymbol{u}_h),\boldsymbol{\eta}_h)\big)\|_{\widetilde{\mathcal{C}}} = \iota^{-1}\|Q_{\Sigma}\boldsymbol{\Phi}-\boldsymbol{\Phi}_h\|_{\widetilde{\mathcal{A}}}. 
\end{equation*}
By the norm equivalence \eqref{normequiv} and the triangle inequality,
\begin{equation}\label{eq:20260615}
\iota\| Q_M\boldsymbol{\varepsilon}(\boldsymbol{u})-(\boldsymbol{\varepsilon}_h(\boldsymbol{u}_h),\boldsymbol{\eta}_h)\|_{1,h}\lesssim \iota^{-1}\|\boldsymbol{\Phi}-Q_{\Sigma}\boldsymbol{\Phi}\|_{\widetilde{\mathcal{A}}}+ \iota^{-1}\|\boldsymbol{\Phi}-\boldsymbol{\Phi}_h\|_{\widetilde{\mathcal{A}}}.
\end{equation}
If $\boldsymbol{\Phi}\in H^2(\Omega;\mathbb{S}\otimes\mathbb{R}^d)$ and $\boldsymbol{u}\in H^3(\Omega;\mathbb{R}^d)$, then \eqref{Qsigma-error} and $\boldsymbol{\Phi}=\iota^2\grad\boldsymbol{\sigma}(\boldsymbol{u})$ yield
\begin{equation*}
\iota^{-1}\|\boldsymbol{\Phi}-Q_{\Sigma}\boldsymbol{\Phi}\|_{\widetilde{\mathcal{A}}}\lesssim \iota^{-1}h|\boldsymbol{\Phi}|_1 \lesssim \iota h|\boldsymbol{\sigma}(\boldsymbol{u})|_2.
\end{equation*}
Combining this estimate with \eqref{eq:20260615} and \eqref{errresult1} proves \eqref{errresult2-hybrid}.

Under Assumption~\ref{ass:uniform-regularity}, the $L^2$-stability of $Q_\Sigma$ and \eqref{Regularity-u}--\eqref{Regularity-divu} give
\begin{equation*}
\iota^{-1}\|\boldsymbol{\Phi}-Q_{\Sigma}\boldsymbol{\Phi}\|_{\widetilde{\mathcal{A}}}\lesssim \iota^{-1}\|\boldsymbol{\Phi}-Q_{\Sigma}\boldsymbol{\Phi}\|_0\lesssim \iota^{-1}\|\boldsymbol{\Phi}\|_0 \lesssim \iota^{1/2}\|\boldsymbol{f}\|_0.
\end{equation*}
Combining this estimate with \eqref{eq:20260615} and \eqref{errresult2} proves \eqref{errresult1-hybrid}.
\end{proof}

% \begin{remark}\rm
% Estimate \eqref{errresult2-hybrid} establishes first-order supercloseness of the discrete weak strain pair $(\boldsymbol{\varepsilon}_h(\boldsymbol{u}_h),\boldsymbol{\eta}_h)$ to the projected exact strain $Q_M\boldsymbol{\varepsilon}(\boldsymbol{u})$. Estimate \eqref{errresult1-hybrid} is uniform with respect to both $\iota$ and $\lambda$.
% \end{remark}

%%%%%%%%%%%%%%%%%%%%%%%%%%%%%%%%%%
%%%%%%% Numerical results
%%%%%%%%%%%%%%%%%%%%%%%%%%%%%%%%%%
\section{Numerical results}\label{sec5}
In this section, we verify the theoretical convergence rates and the robustness of the mixed FEM \eqref{SGE-MFEM} with respect to the size parameter $\iota$ and the Lam\'e constant $\lambda$ through numerical examples in two and three dimensions. Let $\Omega=(0,1)^d$ for $d=2,3$, and use uniform simplicial meshes.
%, see Figure \ref{fig:mesh}. 
Set $\mu=1$.
Introduce the following numerical error
\begin{align*}
\textrm{Err}_1 := \iota^{-1}\|\boldsymbol{\Phi}-\boldsymbol{\Phi}_h\|_{\widetilde{\mathcal{A}}}+{\interleave I_h^{\rm CR}\boldsymbol{u}-\boldsymbol{u}_h\interleave}_{\iota,h}+{\interleave \boldsymbol{u}-\boldsymbol{u}_h\interleave}_{1,h}.
\end{align*}
For the smooth solutions considered below, the triangle inequality, \eqref{IhCRerror1}, and \eqref{stress-controls-u} show that \eqref{errresult1} also controls the $\interleave\cdot\interleave_{1,h}$ component of $\interleave I_h^{\rm CR}\boldsymbol{u}-\boldsymbol{u}_h\interleave_{\iota,h}$. Thus, \eqref{errresult1} predicts the first-order convergence of $\textrm{Err}_1$.

%\begin{figure}[htbp] 
%\centering
%\begin{subfigure}[b]{0.4\textwidth}
%	\centering
%	\includegraphics[width=\textwidth]{figures/2dmesh-8.pdf}
%	\caption{$h=1/8$}
%\end{subfigure}
%\hspace{0.05\textwidth}
%\begin{subfigure}[b]{0.4\textwidth}
%	\centering
%	\includegraphics[width=\textwidth]{figures/3dmesh-2.pdf}
%	\caption{$h=1/2$}
%\end{subfigure}
%\caption{Initial meshes in two dimensions and three dimensions.}
%\label{fig:mesh}
%\end{figure}
\subsection{Numerical results without boundary layers}
We first consider exact solutions that are divergence-free and do not exhibit boundary layers. The right-hand side $\boldsymbol{f}$ is computed from \eqref{SGE0} and is independent of the Lam\'{e} constant $\lambda$. The mesh size $h$ is refined from $2^{-3}$ to $2^{-7}$ in 2D, and from $2^{-1}$ to $2^{-5}$ in 3D. The size parameter $\iota$ takes $10^{-1}$ and $10^{-6}$, while $\lambda$ takes $1$ and $10^6$.
\begin{example}\label{example1}
\normalfont
In two dimensions, the exact displacement field is given by
$$
\boldsymbol{u}=
\left(
\begin{matrix}
	\sin^3(\pi x)\sin(2\pi y)\sin(\pi y)\\
	-\sin^3(\pi y)\sin(2\pi x)\sin(\pi x)
\end{matrix}
\right).
$$
The numerical errors $\textrm{Err}_1$ are listed in Table~\ref{table12}. We observe from Table~\ref{table12} that  $\textrm{Err}_1= \mathcal{O}(h)$ uniformly with respect to both $\iota$ and $\lambda$, confirming the optimal convergence rate predicted by the theoretical result \eqref{errresult1}.

\begin{table}
	\centering
	\caption{$\textrm{Err}_1$ of the mixed method (\ref{SGE-MFEM}) for Example \ref{example1} in 2D.  }
	\vspace{-1.0em}
	\label{table12}
	\begin{tabular}{ccccccc}
		\toprule
		$\lambda$ & $\iota\backslash h$ & 1/8 & 1/16 & 1/32 & 1/64 & 1/128 \\
		\midrule
		\multirow{4}{*}{$1$}
		& $10^{-1}$ & 1.935e+00 & 8.033e-01 & 3.380e-01 & 1.510e-01 & 7.089e-02 \\
		& rate & & 1.27 & 1.25 & 1.16 & 1.09 \\
		& $10^{-6}$ & 4.550e-01 & 1.971e-01 & 9.332e-02 & 4.594e-02 & 2.288e-02 \\
		& rate & & 1.21 & 1.08 & 1.02 & 1.01 \\
		\midrule
		\multirow{4}{*}{$10^{6}$}
		& $10^{-1}$ & 1.927e+00 & 8.248e-01 & 3.472e-01 & 1.538e-01 & 7.164e-02 \\
		& rate & & 1.22 & 1.25 & 1.17 & 1.10 \\
		& $10^{-6}$ & 4.187e-01 & 1.724e-01 & 7.854e-02 & 3.813e-02 & 1.891e-02 \\
		& rate & & 1.28 & 1.13 & 1.04 & 1.01 \\
		\bottomrule
	\end{tabular}
	\vspace{1.0em}
\end{table}

\end{example}

\begin{example}\label{example2}
\normalfont
In three dimensions, the solution takes the form
$$
\boldsymbol{u}=
\left(
\begin{matrix}
	2\sin^3(\pi x)\sin(2\pi y)\sin(\pi y)\sin(2\pi z)\sin(\pi z)\\
	-\sin^3(\pi y)\sin(2\pi z)\sin(\pi z)\sin(2\pi x)\sin(\pi x)\\
	-\sin^3(\pi z)\sin(2\pi x)\sin(\pi x)\sin(2\pi y)\sin(\pi y)
\end{matrix}
\right).
$$
Table~\ref{table12-3d} shows that $\textrm{Err}_1 = \mathcal{O}(h)$ uniformly with respect to both $\iota$ and $\lambda$, in agreement with the theoretical estimate \eqref{errresult1}.
%where a slight reduction in the observed convergence rate occurs.
%A similar phenomenon is seen in the 2D case with $\iota = 1$ on relatively coarse meshes. This suggests that finer discretizations may be required to attain the optimal rate when $\iota$ is relatively large.
%
%Furthermore, after increasing the penalty parameter to $\eta = 10$, Table~\ref{table12-eta10} shows that the method still achieves first-order convergence, demonstrating that the method is both robust and optimally convergent.
\begin{table}
	\centering
	\caption{$\textrm{Err}_1$ of the mixed method (\ref{SGE-MFEM}) for Example \ref{example2} in 3D.  }
	\vspace{-1.0em}
	\label{table12-3d}
	\begin{tabular}{ccccccc}
		\toprule
		$\lambda$ & $\iota\backslash h$ & 1/2 & 1/4 & 1/8 & 1/16 & 1/32\\
		\midrule
		\multirow{4}{*}{$1$}
		& $10^{-1}$ & 7.493e+00 & 6.230e+00 & 3.097e+00 & 1.446e+00 & 6.905e-01\\
		& rate &  & 0.27 & 1.01 & 1.10 & 1.07\\
		& $10^{-6}$ & 3.074e+00 & 2.382e+00 & 1.305e+00 & 6.617e-01 & 3.314e-01\\
		& rate &  & 0.37 & 0.87 & 0.98 & 1.00\\
		\midrule
		\multirow{4}{*}{$10^{6}$}
		& $10^{-1}$ & 7.417e+00 & 6.151e+00 & 3.088e+00 & 1.456e+00 & 6.962e-01 \\
		& rate &  & 0.27 & 0.99 & 1.08 & 1.06 \\
		& $10^{-6}$ & 3.060e+00 & 2.355e+00 & 1.288e+00 & 6.518e-01 & 3.260e-01\\
		& rate &  & 0.38 & 0.87 & 0.98 & 1.00\\
		\bottomrule
	\end{tabular}
	\vspace{1.0em}
\end{table}
%\begin{table}
%	\centering
%	\caption{$\textrm{Err}_1$ of the mixed method (\ref{SGE-MFEM}) for Example \ref{example1} with $\iota=1$, $\eta=10$ in 3D.  }
%	\vspace{-1.0em}
%	\label{table12-eta10}
%	\begin{tabular}{cccccc}
%		\toprule
%		$\lambda\backslash h$ & 1/2 & 1/4 & 1/8 & 1/16 & 1/32\\
%		\midrule
%		1 & 5.654e+01 & 5.339e+01 & 2.839e+01 & 1.562e+01 & 7.444e+00 \\
%		rate & & 0.08 & 0.91 & 0.86 & 1.07 \\
%		\midrule
%		$10^{6}$ & 5.519e+01 & 5.186e+01 & 2.741e+01 & 1.523e+01 & 7.571e+00\\
%		rate &  & 0.09 & 0.92 & 0.85 & 1.01 \\
%		\bottomrule
%	\end{tabular}
%	\vspace{1.0em}
%\end{table}
\end{example}
\subsection{Numerical results with boundary layers}
Next, we verify the convergence of the discrete method \eqref{SGE-MFEM} with boundary layers. We adopt divergence-free solutions of the linear elasticity problem \eqref{SGElinear}, which induce strong boundary layers as $\iota \rightarrow 0$. The right-hand side $\boldsymbol{f}$, computed from \eqref{SGElinear}, is used as the forcing term in \eqref{SGE0}, and notably, it is independent of both the size parameter $\iota$ and the Lam\'{e} constant $\lambda$. For the convergence tests, we take $\iota=10^{-6},10^{-8}$ and $\lambda=1,10^6$, using the same mesh refinement levels as in Examples~\ref{example1} and~\ref{example2}.
\begin{example}\label{example3}
\normalfont

In two dimensions, the linear elasticity solution is given by
$$
\boldsymbol{u}_0=
\left(
\begin{matrix}
	(e^{\cos(2\pi x)}-e)\sin(2\pi y)e^{\cos(2\pi y)}\\
	-(e^{\cos(2\pi y)}-e)\sin(2\pi x)e^{\cos(2\pi x)}
\end{matrix}
\right).
$$
%Next, we assess the performance of the method for problems with strong boundary layers, where the exact solution of the SGE model is variable but is expected to develop sharp gradients as $\iota \rightarrow 0$. We take the right side term $\boldsymbol{f}$ computed from (\ref{SGElinear}) as the right side function of problem \eqref{SGE0}. Take $\iota = 10^{-4}, 10^{-6}, 10^{-8}$ and $\lambda = 1,10^6$.
The error ${\interleave \boldsymbol{u}_0-\boldsymbol{u}_h\interleave}_{1,h}$, shown in Table~\ref{table34}, demonstrates uniform convergence of order $\mathcal{O}(h)$, in line with the theoretical estimate \eqref{errresult3}, despite the presence of strong boundary layers. Therefore, the discrete method (\ref{SGE-MFEM}) is not only robust with respect to the size parameter $\iota$ and the Lam\'{e} constant $\lambda$, but also optimal.

To further illustrate the boundary layer behavior, we visualize the Frobenius norm of $\iota^{-2}\boldsymbol{\Phi}_h$ with $h=\frac{1}{128}$. Although the solution $\boldsymbol{u}$ to problem \eqref{SGE0} with this right-hand side is not known explicitly, $\iota^{-2}\boldsymbol{\Phi}_h$ serves as a numerical approximation to $\grad\boldsymbol{\sigma}(\boldsymbol{u})$.
%Since $\boldsymbol{\Phi}_h$ is piecewise affine, we evaluate $\iota^{-2}\boldsymbol{\Phi}_h|_T$ separately on each element $T$ and plot its Frobenius norm as a discontinuous piecewise-polynomial surface; no nodal averaging or additional projection is applied across element interfaces.
Since $\boldsymbol{\Phi}_h$ is piecewise affine, we evaluate the pointwise Frobenius norm of $\iota^{-2}\boldsymbol{\Phi}_h|_T$ separately on each element $T$ and plot it as a discontinuous surface; no nodal averaging or additional projection is applied across element interfaces.
Figures~\ref{fig1} and~\ref{fig2} show that the boundary layers become increasingly pronounced as $\iota$ decreases.

\begin{table}
	\centering
	\caption{${\interleave \boldsymbol{u}_0-\boldsymbol{u}_h\interleave}_{1,h}$ of the mixed method (\ref{SGE-MFEM}) for Example~\ref{example3} in 2D.}
	\vspace{-1.0em}
	\label{table34}
	\begin{tabular}{ccccccc}
		\toprule
		$\lambda$ & $\iota\backslash h$ & 1/8 & 1/16 & 1/32 & 1/64 & 1/128 \\
		\midrule
		\multirow{4}{*}{$1$} 
		& $10^{-6}$ & 2.706e+00 & 1.169e+00 & 5.351e-01 & 2.591e-01 & 1.282e-01 \\
		& rate & & 1.21 & 1.13 & 1.05 & 1.02 \\
		& $10^{-8}$ & 2.706e+00 & 1.169e+00 & 5.351e-01 & 2.591e-01 & 1.282e-01 \\
		& rate & & 1.21 & 1.13 & 1.05 & 1.02 \\
		\midrule
		\multirow{4}{*}{$10^{6}$} 
		& $10^{-6}$ & 2.493e+00 & 1.016e+00 & 4.404e-01 & 2.088e-01 & 1.026e-01 \\
		& rate & & 1.29 & 1.21 & 1.08 & 1.02 \\
		& $10^{-8}$ & 2.493e+00 & 1.016e+00 & 4.404e-01 & 2.088e-01 & 1.026e-01 \\
		& rate & & 1.29 & 1.21 & 1.08 & 1.02 \\
		\bottomrule
	\end{tabular}
	\vspace{1.0em}
\end{table}

%\begin{figure}[htbp]
%	\centering
%	
%	% row1
%	\begin{subfigure}[b]{0.3\textwidth}
%		\centering
%		\includegraphics[width=\textwidth]{figures/bl00.pdf}
%		\caption{$\iota=1$}
%		\label{image1-1}
%	\end{subfigure}
%	\hfill
%	\begin{subfigure}[b]{0.3\textwidth}
%		\centering
%		\includegraphics[width=\textwidth]{figures/bl-10.pdf}
%		\caption{$\iota=10^{-1}$}
%		\label{image2-1}
%	\end{subfigure}
%	\hfill
%	\begin{subfigure}[b]{0.3\textwidth}
%		\centering
%		\includegraphics[width=\textwidth]{figures/bl-20.pdf} 
%		\caption{$\iota=10^{-2}$}
%		\label{image3-1}
%	\end{subfigure}
%	%\vspace{1cm}
%	% row2
%	\begin{subfigure}[b]{0.3\textwidth}
%		\centering
%		\includegraphics[width=\textwidth]{figures/bl-30.pdf}
%		\caption{$\iota=10^{-3}$}
%		\label{image4-1}
%	\end{subfigure}
%	\hfill
%	\begin{subfigure}[b]{0.3\textwidth}
%		\centering
%		\includegraphics[width=\textwidth]{figures/bl-40.pdf}
%		\caption{$\iota=10^{-4}$}
%		\label{image5-1}
%	\end{subfigure}
%	\hfill
%	\begin{subfigure}[b]{0.3\textwidth}
%		\centering
%		\includegraphics[width=\textwidth]{figures/bl-50.pdf}
%		\caption{$\iota=10^{-5}$}
%		\label{image6-1}
%	\end{subfigure}
%	
%	\caption{The Frobenius norm of $\iota^{-2}\boldsymbol{\Phi}_h$ with different $\iota$ and $\lambda=1$.}
%	\label{six_images_1}
%\end{figure}
\begin{figure}[htbp]
	\centering
	\captionsetup[subfigure]{skip=2pt}
	\setlength{\tabcolsep}{2pt}
	
	\begin{tabular}{ccc}
		
		\begin{subfigure}{0.32\textwidth}
			\centering
			\includegraphics[width=\linewidth]{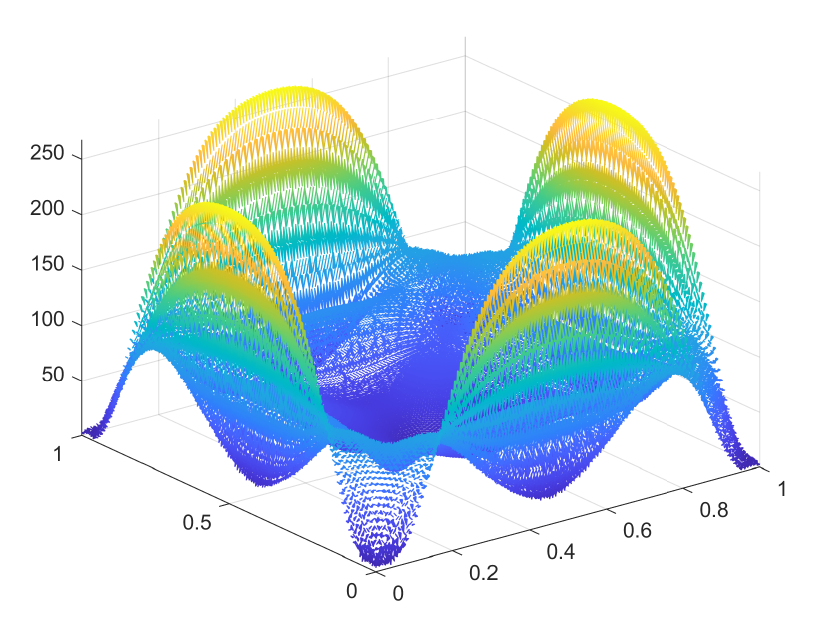}
			\caption{$\iota=10^{-1}$}
		\end{subfigure}
		\hfill 
		\begin{subfigure}{0.32\textwidth}
			\centering
			\includegraphics[width=\linewidth]{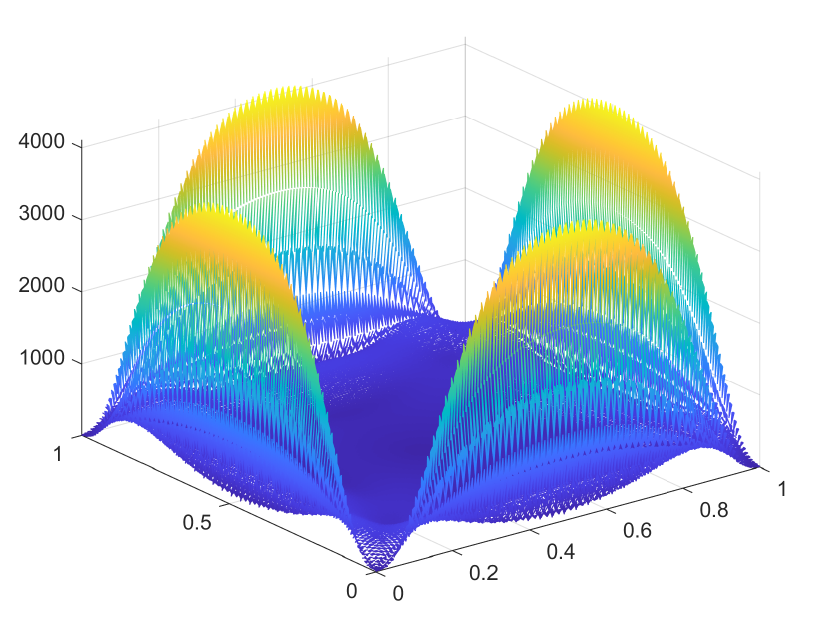}
			\caption{$\iota=10^{-2}$}
		\end{subfigure}
		\hfill 
		\begin{subfigure}{0.32\textwidth}
			\centering
			\includegraphics[width=\linewidth]{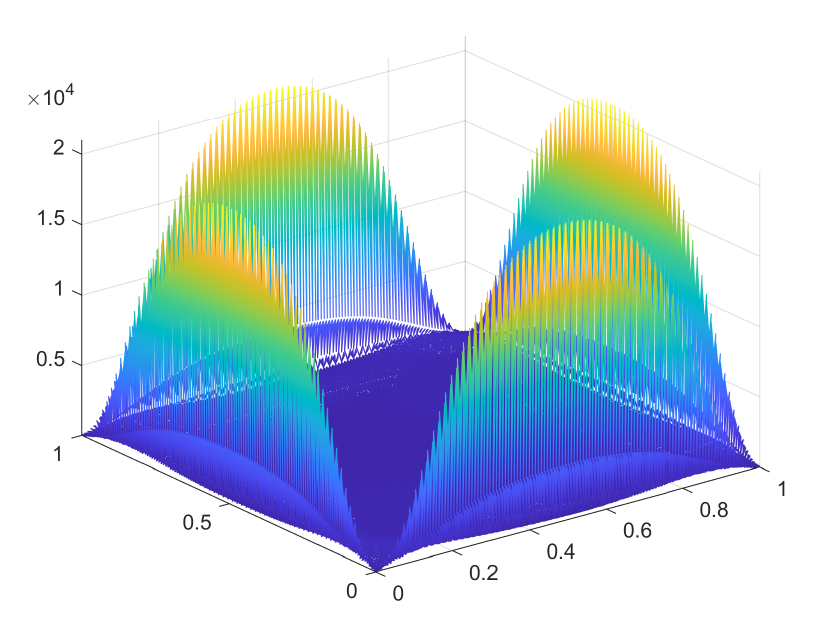}
			\caption{$\iota=10^{-3}$}
		\end{subfigure}
		
	\end{tabular}
\caption{Pointwise Frobenius norm of $\iota^{-2}\boldsymbol{\Phi}_h$, plotted elementwise for Example~\ref{example3}, with $h=1/128$, $\lambda=1$, and several values of $\iota$.}
	\label{fig1}
\end{figure}

%\begin{figure}[htbp]
%	\centering
%	
%	% row1
%	\begin{subfigure}[b]{0.3\textwidth}
%		\centering
%		\includegraphics[width=\textwidth]{figures/bl06.pdf}
%		\caption{$\iota=1$}
%		\label{image1-2}
%	\end{subfigure}
%	\hfill
%	\begin{subfigure}[b]{0.3\textwidth}
%		\centering
%		\includegraphics[width=\textwidth]{figures/bl-16.pdf}
%		\caption{$\iota=10^{-1}$}
%		\label{image2-2}
%	\end{subfigure}
%	\hfill
%	\begin{subfigure}[b]{0.3\textwidth}
%		\centering
%		\includegraphics[width=\textwidth]{figures/bl-26.pdf} 
%		\caption{$\iota=10^{-2}$}
%		\label{image3-2}
%	\end{subfigure}
%	%\vspace{1cm}
%	% row2
%	\begin{subfigure}[b]{0.3\textwidth}
%		\centering
%		\includegraphics[width=\textwidth]{figures/bl-36.pdf}
%		\caption{$\iota=10^{-3}$}
%		\label{image4-2}
%	\end{subfigure}
%	\hfill
%	\begin{subfigure}[b]{0.3\textwidth}
%		\centering
%		\includegraphics[width=\textwidth]{figures/bl-46.pdf}
%		\caption{$\iota=10^{-4}$}
%		\label{image5-2}
%	\end{subfigure}
%	\hfill
%	\begin{subfigure}[b]{0.3\textwidth}
%		\centering
%		\includegraphics[width=\textwidth]{figures/bl-56.pdf}
%		\caption{$\iota=10^{-5}$}
%		\label{image6-2}
%	\end{subfigure}
%	
%	\caption{The Frobenius norm of $\iota^{-2}\boldsymbol{\Phi}_h$ with different $\iota$ and $\lambda=10^6$.}
%	\label{six_images_2}
%\end{figure}
\begin{figure}[htbp]
	\centering
	\captionsetup[subfigure]{skip=2pt}
	\setlength{\tabcolsep}{2pt}
	
	\begin{tabular}{ccc}
		
		\begin{subfigure}{0.32\textwidth}
			\centering
			\includegraphics[width=\linewidth]{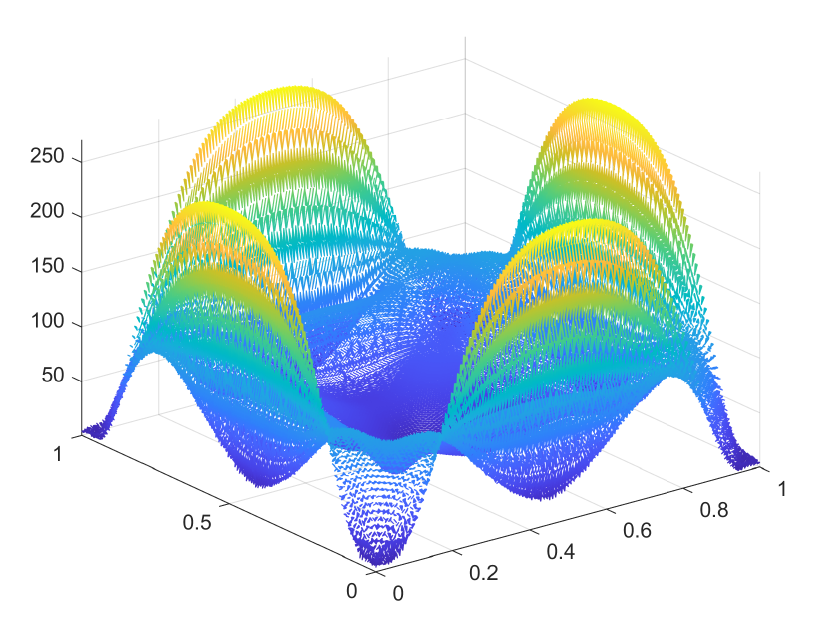}
			\caption{$\iota=10^{-1}$}
		\end{subfigure}
		\hfill 
		\begin{subfigure}{0.32\textwidth}
			\centering
			\includegraphics[width=\linewidth]{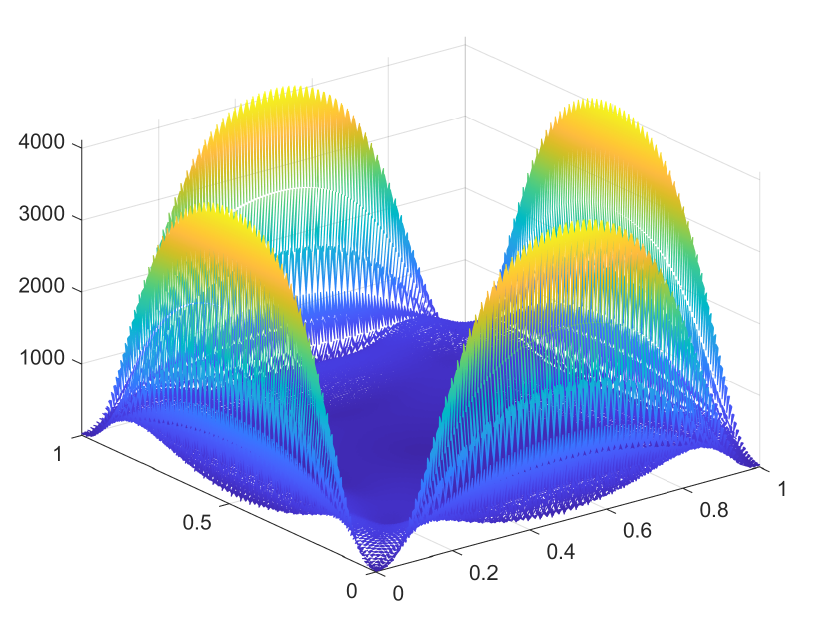}
			\caption{$\iota=10^{-2}$}
		\end{subfigure}
		\hfill 
		\begin{subfigure}{0.32\textwidth}
			\centering
			\includegraphics[width=\linewidth]{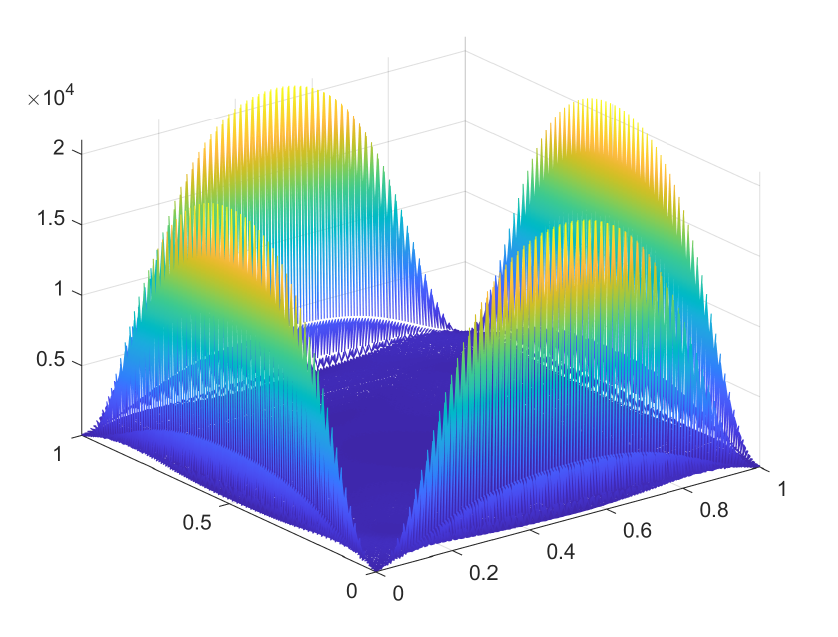}
			\caption{$\iota=10^{-3}$}
		\end{subfigure}
		
	\end{tabular}
\caption{Pointwise Frobenius norm of $\iota^{-2}\boldsymbol{\Phi}_h$, plotted elementwise for Example~\ref{example3}, with $h=1/128$, $\lambda=10^6$, and several values of $\iota$.}
	\label{fig2}
\end{figure}

\end{example}

\begin{example}\label{example4}
\normalfont
In three dimensions, we take the linear elasticity solution
$$
\boldsymbol{u}_0=
\left(
\begin{matrix}
	2x^2(1-x)^2y(1-y)(1-2y)z(1-z)(1-2z)\\
	-y^2(1-y)^2x(1-x)(1-2x)z(1-z)(1-2z)\\
	-z^2(1-z)^2x(1-x)(1-2x)y(1-y)(1-2y)
\end{matrix}
\right).
$$
As shown in Table~\ref{table34-3d}, the error ${\interleave \boldsymbol{u}_0-\boldsymbol{u}_h\interleave}_{1,h}= \mathcal{O}(h)$, which further confirms the robustness and optimality of the method \eqref{SGE-MFEM} with respect to both $\iota$ and $\lambda$.
\begin{table}
	\centering
	\caption{${\interleave \boldsymbol{u}_0-\boldsymbol{u}_h\interleave}_{1,h}$ of the mixed method (\ref{SGE-MFEM}) for Example~\ref{example4} in 3D.}
	\vspace{-1.0em}
	\label{table34-3d}
	\begin{tabular}{ccccccc}
		\toprule
		$\lambda$ & $\iota\backslash h$ & 1/2 & 1/4 & 1/8 & 1/16 & 1/32 \\
		\midrule
		\multirow{4}{*}{$1$} 
		& $10^{-6}$ & 3.765e-03 & 2.634e-03 & 1.427e-03 & 7.208e-04 & 3.603e-04\\
		& rate &  & 0.52 & 0.88 & 0.99 & 1.00\\
		& $10^{-8}$ & 3.765e-03 & 2.634e-03 & 1.427e-03 & 7.208e-04 & 3.603e-04\\
		& rate &  & 0.52 & 0.88 & 0.99 & 1.00\\
		\midrule
		\multirow{4}{*}{$10^{6}$} 
		& $10^{-6}$ & 3.730e-03 & 2.607e-03 & 1.409e-03 & 7.108e-04 & 3.549e-04\\
		& rate &  & 0.52 & 0.89 & 0.99 & 1.00\\
		& $10^{-8}$ & 3.730e-03 & 2.607e-03 & 1.409e-03 & 7.108e-04 & 3.549e-04\\
		& rate &  & 0.52 & 0.89 & 0.99 & 1.00\\
		%& $10^{-8}$ & 4.041e-03 & 2.935e-03 & 1.682e-03 & 8.697e-04 & 4.382e-04\\
		%& rate &  & 0.46 & 0.80 & 0.95 & 0.99\\
		\bottomrule
	\end{tabular}
	\vspace{1.0em}
\end{table}
\end{example}
\bibliographystyle{abbrv} % F6+F8+F6+F8
\bibliography{ref} 
\end{document}